\documentclass[reqno]{amsart}

\usepackage[centertags]{amsmath}
\usepackage[colorlinks]{hyperref}
\usepackage{amsfonts}
\usepackage{MnSymbol}
\usepackage{amsthm}
\usepackage{newlfont}
\usepackage{amscd}
\usepackage{amsmath}
\usepackage{enumerate}
\usepackage[all,2cell]{xy}
\UseAllTwocells
\input xy
\xyoption{2cell}
\xyoption{all}
\usepackage{verbatim}
\usepackage{eucal}
\usepackage{xpatch}
\usepackage{graphicx}

\newcommand{\poc}[1][dr]{\save*!/#1+1.6pc/#1:(1,-1)@^{|-}\restore}

\newcommand{\NN}{\mathbb{N}}

\newcommand{\ZZ}{\mathbb{Z}}

\newcommand{\RR}{\mathbb{R}}
\newcommand{\EE}{\mathbb{E}}

\def\C{\mathcal{C}}
\def\D{\mathcal{D}}
\def\G{\mathcal{G}}

\def\F{\mathcal{F}}

\def\M{\mathcal{M}}
\def\N{\mathcal{N}}

\def\I{\mathcal{I}}

\def\T{\mathcal{T}}
\def\P{\mathcal{P}}
\def\Q{\mathcal{Q}}

\renewcommand{\L}{\mathcal{L}}
\def\U{\mathcal{U}}

\def\O{\mathcal{O}}

\def\bS{\mathbf{S}}

\def\rN{\mathrm{N}}

\def\K{\mathcal{K}}
\def\X{\mathcal{X}}

\def\sig{{\sigma}}

\def\Om{{\Omega}}

\def\Sig{{\Sigma}}

\def\vphi{\varphi}

\newcommand{\HSwarrow}{\kern0.05ex\vcenter{\hbox{\Huge\ensuremath{\Swarrow}}}\kern0.05ex}
\newcommand{\hSwarrow}{\kern0.05ex\vcenter{\hbox{\huge\ensuremath{\Swarrow}}}\kern0.05ex}
\newcommand{\LLSwarrow}{\kern0.05ex\vcenter{\hbox{\LARGE\ensuremath{\Swarrow}}}\kern0.05ex}
\newcommand{\LSwarrow}{\kern0.05ex\vcenter{\hbox{\Large\ensuremath{\Swarrow}}}\kern0.05ex}

\newcommand{\HSearrow}{\kern0.05ex\vcenter{\hbox{\Huge\ensuremath{\Searrow}}}\kern0.05ex}
\newcommand{\hSearrow}{\kern0.05ex\vcenter{\hbox{\huge\ensuremath{\Searrow}}}\kern0.05ex}
\newcommand{\LLSearrow}{\kern0.05ex\vcenter{\hbox{\LARGE\ensuremath{\Searrow}}}\kern0.05ex}
\newcommand{\LSearrow}{\kern0.05ex\vcenter{\hbox{\Large\ensuremath{\Searrow}}}\kern0.05ex}

\newcommand{\HDownarrow}{\kern0.05ex\vcenter{\hbox{\Huge\ensuremath{\Downarrow}}}\kern0.05ex}
\newcommand{\hDownarrow}{\kern0.05ex\vcenter{\hbox{\huge\ensuremath{\Downarrow}}}\kern0.05ex}
\newcommand{\LLDownarrow}{\kern0.05ex\vcenter{\hbox{\LARGE\ensuremath{\Downarrow}}}\kern0.05ex}
\newcommand{\LDownarrow}{\kern0.05ex\vcenter{\hbox{\Large\ensuremath{\Downarrow}}}\kern0.05ex}

\newcommand{\HUparrow}{\kern0.05ex\vcenter{\hbox{\Huge\ensuremath{\Uparrow}}}\kern0.05ex}
\newcommand{\hUparrow}{\kern0.05ex\vcenter{\hbox{\huge\ensuremath{\Uparrow}}}\kern0.05ex}
\newcommand{\LLUparrow}{\kern0.05ex\vcenter{\hbox{\LARGE\ensuremath{\Uparrow}}}\kern0.05ex}
\newcommand{\LUparrow}{\kern0.05ex\vcenter{\hbox{\Large\ensuremath{\Uparrow}}}\kern0.05ex}

\newtheorem{thm}{Theorem}[subsection]
\newtheorem{cor}[thm]{Corollary}
\newtheorem{lem}[thm]{Lemma}
\newtheorem{pro}[thm]{Proposition}

\makeatletter
\@addtoreset{thm}{section}
\makeatother

\numberwithin{equation}{subsection}
\numberwithin{thm}{subsection}

\theoremstyle{definition}
\newtheorem{define}[thm]{Definition}

\newtheorem{example}[thm]{Example}
\newtheorem{examples}[thm]{Examples}
\newtheorem{defn}[thm]{Definition}
\newtheorem{cons}[thm]{Construction}

\newtheorem{notn}[thm]{Notation}

\theoremstyle{remark}
\newtheorem{rem}[thm]{Remark}

\DeclareMathOperator{\hocolim}{hocolim}
\DeclareMathOperator{\colim}{colim}

\DeclareMathOperator{\id}{id}
\DeclareFontFamily{OT1}{pzc}{}
\DeclareFontShape{OT1}{pzc}{m}{it}{<-> s * [1.10] pzcmi7t}{}
\DeclareMathAlphabet{\mathpzc}{OT1}{pzc}{m}{it}
\DeclareMathOperator{\cof}{cof}

\DeclareMathOperator{\ad}{ad}

\DeclareMathOperator{\Alg}{Alg}
\DeclareMathOperator{\Ass}{Ass}

\DeclareMathOperator{\op}{op}
\DeclareMathOperator{\Map}{Map}

\DeclareMathOperator{\inj}{inj}
\DeclareMathOperator{\proj}{proj}

\DeclareMathOperator{\Set}{Set}
\DeclareMathOperator{\RelCat}{RelCat}

\DeclareMathOperator{\Cat}{Cat}

\DeclareMathOperator{\Mod}{Mod}

\DeclareMathOperator{\Fun}{Fun}

\DeclareMathOperator{\Ob}{Ob}
\DeclareMathOperator{\Sp}{Sp}
\DeclareMathOperator{\Ho}{Ho}

\DeclareMathOperator{\Op}{Op}
\DeclareMathOperator{\Ab}{Ab}
\DeclareMathOperator{\Env}{Env}

\DeclareMathOperator{\Tw}{Tw}

\DeclareMathOperator{\aug}{aug}
\DeclareMathOperator{\SymSeq}{SymSeq}
\DeclareMathOperator{\Sub}{Sub}
\DeclareMathOperator{\Dec}{Dec}
\DeclareMathOperator{\Aut}{Aut}
\DeclareMathOperator{\coeq}{coeq}

\DeclareMathOperator{\Spectra}{Spectra}
\DeclareMathOperator{\Com}{Com}

\DeclareMathOperator{\Reedy}{Reedy}
\DeclareMathOperator{\Grp}{Grp}
\DeclareMathOperator{\Ring}{Ring}

\def\x{\overset}

\def\Lie{\textrm{Lie\,}}

\def\Hom{\textrm{Hom}}

\def\inv{\textup{inv}}

\def\Dec{\mathrm{Dec}}

\newcommand{\tgpd}{\kern0.05ex\vcenter{\hbox{\footnotesize\ensuremath{2}}}\kern0.05ex\mathcal{G}pd} 

\def\rar{\rightarrow}

\def\lrar{\longrightarrow}

\def\hrar{\hookrightarrow}

\newcommand{\adj}{\mathrel{\substack{\longrightarrow \\[-.6ex] \x{\upvdash}{\longleftarrow}}}}

\def\ovl{\overline}

\def\uline{\underline}

\makeatletter
\def\maketag@@@#1{\hbox{\m@th\normalfont\normalsize#1}}
\makeatother

\title{Tangent categories of algebras over operads}

\author{Yonatan Harpaz}
\email{harpaz@math.univ-paris13.fr}
\address{Institut Galil\'ee \\
Universit\'e Paris 13 \\
99 avenue J.B. Cl\'ement \\
93430 Villetaneuse \\
France}

\author{Joost Nuiten}
\email{j.j.nuiten@uu.nl}
\address{Mathematical Institute\\ Utrecht University\\ P.O. Box 80010\\ 3508 TA Utrecht\\ The
Netherlands.}
\author{Matan Prasma}
\email{mtnprsm@gmail.com}
\address{Faculty of Mathematics\\ University of Regensburg\\ Universitatsstrase 31, 93040\\ Germany.}

\date{}

\begin{document}

\begin{abstract}
Associated to a presentable $\infty$-category $\C$ and an object $X \in \C$ is the tangent $\infty$-category $\T_X\C$, consisting of parameterized spectrum objects over $X$. This gives rise to a cohomology theory, called Quillen cohomology, whose category of coefficients is $\T_X\C$. When $\C$ consists of algebras over a nice $\infty$-operad in a stable $\infty$-category, $\T_X\C$ is equivalent to the $\infty$-category of operadic modules, by work of Basterra--Mandell, Schwede and Lurie. 
In this paper we develop the model-categorical counterpart of this identification and extend it to the case of algebras over an enriched operad, taking values in a model category which is not necessarily stable. This extended comparison can be used, for example, to identify the cotangent complex of enriched categories, an application we take up in a subsequent paper.
\end{abstract}

\maketitle

\tableofcontents

\section{Introduction}\label{s:intro}

A ubiquitous theme in mathematics is the contrast between linear and non-linear structures. In algebraic settings, linear objects such as vector spaces, abelian groups, and modules tend to have a highly structured and accessible theory, while non-linear objects, such as groups, rings, or algebraic varieties are often harder to analyze. 
Non-linear objects often admit interesting \textbf{linear invariants} which are fairly computable and easy to manipulate. Homological algebra then typically enters the picture, extending a given invariant to a collection of derived ones.

To streamline this idea one would like to have a formal framework to understand what linear objects are and how one can ``linearize'' a given non-linear object. One way to do so is the following. Let $\Ab$ denote the category of abelian groups. A locally presentable category $\C$ is called \textbf{additive} if it is tensored over $\Ab$. We note that in this case the tensoring is essentially unique and induces a natural enrichment of $\C$ in $\Ab$. If $\D$ is a locally presentable category then there exists a universal additive category $\Ab(\D)$ receiving a colimit preserving functor $\ZZ:\D \lrar \Ab(\D)$. The category $\Ab(\D)$ can be described explicitly as the category of abelian group objects in $\D$, namely, objects $M \in \D$ equipped with maps $u: \ast_{\D} \lrar M$, $m: M \times M \lrar M$ and $\inv: M \lrar M$ satisfying (diagramatically) all the axioms of an abelian group. We may then identify $\ZZ:\D \lrar \Ab(\D)$ with the functor which sends $A$ to the free abelian group $\ZZ A$ generated from $A$, or the \textbf{abelianization} of $A$.

When studying maps $f: B\lrar A$ one is often interested in linear invariants of $B$ \textbf{over $A$}. A formal procedure to obtain this was developed by Beck in~\cite{Bec67}, where he defined the notion of a \textbf{Beck module} over an object $A$ (say, in a locally presentable category $\D$) to be an abelian group object of the slice category $\D_{/A}$. Simple as it is, this definition turns out to capture many well-known instances of ``linear objects over a fixed object $A$''. For example, if $G$ is a group and $M$ is a $G$-module then the semi-direct product $M \rtimes G$ carries a natural structure of an abelian group object in $\Grp_{/G}$. One can then show that the association $M \mapsto M \rtimes G$ determines an equivalence between the category of $G$-modules and the category of abelian group objects in $\Grp_{/G}$.  
If $\D = \Ring$ is the category of associative unital rings then one may replace the formation of semi-direct products with that of \textbf{square-zero extensions}, yielding an equivalence between the notion of a Beck module over a ring $R$ and the notion of an $R$-bimodule  
(see~\cite{Qui70}). When $R$ is a commutative ring the corresponding notion of a Beck module reduces to the usual notion of an $R$-module. 

In the realm of algebraic topology, one linearizes spaces by evaluating homology theories on them. This approach is closely related to the approach of Beck: indeed, by the classical Dold-Thom theorem one may identify the ordinary homology groups of a space $X$ with the homotopy groups of the free abelian group generated from $X$ (considered, for example, as a simplicial abelian group). The quest for more refined invariants has led to the consideration of generalized homology theories and their classification via homotopy types of \textbf{spectra}. The extension of homological invariants from ordinary homology to generalized homology therefore highlights spectra as a natural extension of the notion of ``linearity'' provided by abelian groups, replacing additivity with \textbf{stability}. This has the favorable consequence that kernels and cokernels of maps become equivalent up to a shift. Using stability as the fundamental form of linearity is also the starting point for the theory of Goodwillie calculus, which extends the notion of stability to give meaningful analogues to higher order approximations, derivatives and Taylor series for functors between $\infty$-categories. 

Replacing the category of abelian groups with the $\infty$-category of spectra means we should replace the notion of an additive category with that of a \textbf{stable $\infty$-category}. 
Thus, instead of associating to a locally presentable category $\D$ the additive category $\Ab(\D)$ of abelian group objects in $\D$, we now associate to a presentable $\infty$-category $\D$ the $\infty$-category $\Sp(\D)$ of \textbf{spectrum objects in $\D$}, which is the universal stable presentable $\infty$-category receiving a colimit preserving functor $\Sig^{\infty}_+:\D \lrar \Sp(\D)$.

The two forms of linearization given by Beck modules and spectra were brought together in~\cite[\S 7.3]{Lur14} under the framework of the \textbf{abstract cotangent complex formalism}.
Given a presentable $\infty$-category $\D$ and an object $A \in \D$, one may define the analogue of a Beck module to be a spectrum object in the slice $\infty$-category $\D_{/A}$. 
As in~\cite{Lur14}, we will refer to $\Sp(\D_{/A})$ as the \textbf{tangent $\infty$-category} at $A$, and denote it by $\T_A\D$. 
These various tangent $\infty$-categories can be assembled into a global object, which is then known as the tangent bundle $\infty$-category $\T\D$.

The cotangent complex formalism allows one to produce cohomological invariants of a given object $A \in \D$ in a universal way. The resulting cohomology groups are known as \textbf{Quillen cohomology groups}, and take their coefficients in the tangent $\infty$-category $\T_A\D$ (see~\cite[\S 2.2]{HNP16} for a more precise comparison with the classical definition of Quillen cohomology via abelianization). In order to study Quillen cohomology effectively one should therefore understand the various tangent $\infty$-categories $\T_A\D$ in reasonably concrete terms. 

One of the main theorems of~\cite[\S 7.3]{Lur14} identifies the tangent $\infty$-categories  $\T_{A}(\Alg_\P(\C))$ of algebras in a presentable stable $\infty$-category $\C$ over a given (unital, coherent) $\infty$-operad $\P$ with the corresponding operadic module $\infty$-categories $\Mod_A(\C)$. Earlier results along these lines were obtained in~\cite{Sch97} and \cite{BM05}.  
For example, the above results identify the tangent $\infty$-category to $\EE_\infty$-ring spectra at a given $\EE_\infty$-ring spectrum $A$ with the $\infty$-category of $A$-module spectra.
This allows one to identify the (abstract) Quillen cohomology of an $\EE_\infty$-ring spectrum with the corresponding topological Andr\'e-Quillen cohomology.

Our main motivation in this paper is to generalize these results to the setting where the operadic algebras take values in an $\infty$-category which is \textbf{not necessarily stable}. In this case, the objects of the tangent $\infty$-categories can be thought of as ``twisted'' modules (see Corollary~\ref{c:sum-comparison-2-into}). 

For various reasons we found it convenient to work in the setting of combinatorial model categories, using~\cite{part1} as our model for stabilization (see \S\ref{s:recall}). Our main result can be formulated as follows (see Corollary~\ref{c:sum-comparison} below).
\begin{thm}[see Corollary~\ref{c:sum-comparison}]\label{t:main-intro}
Let $\M$ be a differentiable combinatorial symmetric monoidal model category, $\P$ a colored symmetric operad in $\M$ and $A$ a $\P$-algebra. Then under suitable technical hypotheses the Quillen adjunction
$$
\xymatrix{
\T_A\Alg_{\P}(\M) \ar@<1ex>[r] & \T_A\Mod_A^{\P}(\M) \ar@<1ex>[l]_-{\downvdash}} 
$$
induced by the free-forgetful adjunction is a Quillen equivalence. 
\end{thm}

\begin{rem}
The role of the technical conditions alluded to in Theorem~\ref{t:main-intro} is mostly to ensure that all the relevant model structures exist and are homotopically sound. They hold, for example, when every object in $\M$ is cofibrant and $\P$ is a cofibrant single-colored operad, or when $\M$ is the category of simplicial sets with the Kan-Quillen model structure and $\P$ is an arbitrary cofibrant colored operad (see Remark~\ref{r:fresse}). When the model structures above do not exist, we can still obtain a similar comparison result for the associated relative categories (see Corollary~\ref{c:sum-comparison-model}).
\end{rem}

The main idea behind Theorem~\ref{t:main-intro} becomes most apparent in the case where $A$ is the initial $\P$-algebra. In this case, $\P$-algebras over $A$ are just augmented $\P$-algebras and $A$-modules are augmented algebras over the sub-operad $\P_{\leq 1}\lrar \P$ containing only the $0$- and $1$-ary operations of $\P$. We can reformulate the key principle behind Theorem \ref{t:main-intro} as follows: the process of stabilization is \textbf{insensitive to algebraic operations of arity $\geq 2$}. This is the content of Theorem \ref{t:comp}, which relies on a skeletal filtration discussed in Section \ref{ss:filtration}.

In fact, this principle also implies the theorem for a general $\P$-algebra $A$, by the following two observations:
\begin{enumerate}[(1)]
\item Since spectrum objects are canonically pointed, the tangent category of $\Alg_{\P}(\M)$ at $A$ is equivalent to the tangent category of the under-category $\Alg_{\P}(\M)_{A/}$, at $A$. Since $A$ is the initial object of $\Alg_{\P}(\M)_{A/}$, the tangent category $\T_A\Alg_{\P}(\M)$ is equivalent to the stabilization of augmented objects in $\P$-algebras under $A$. Similarly, $\T_A\Mod^{\P}_A(\M)$ is equivalent to the stabilization of augmented $A$-modules under $A$.
 
\item $\P$-algebras under $A$ are algebras over the \textbf{enveloping operad} $\P^A$ of $A$, while $A$-modules under $A$ are algebras over the sub-operad $\P^A_{\leq 1}\lrar \P^A$. In other words, by replacing $\P$ by $\P^A$ we reduce to the case where $A$ is the initial algebra, as described above.
\end{enumerate}

Under suitable assumptions, the tangent model category at a given operadic algebra $A$ can therefore be identified with the tangent category to $A$ in the model category of $A$-modules. This latter tangent category can be further simplified into something that resembles a functor category with stable codomain. To make this idea precise it is useful to exploit the global point of view obtained by assembling the various tangent categories into a \textbf{tangent bundle}, using the model categorical machinery of~\cite{part1}. 
The final identification of $\T_A\Alg_\P(\M)$ then takes the following form (see Corollary~\ref{c:sum-comparison-2} below):
\begin{cor}\label{c:sum-comparison-2-into}
Let $\M, \P$ and $A$ be as in Theorem~\ref{t:main-intro}. Then we have a natural Quillen equivalence
$$\xymatrix{
\T_A\Alg_{\P}(\M) \ar@<1ex>[r]^-{\sim} & \Fun^{\M}_{/\M}(\P^A_1,\T\M) \ar@<1ex>[l]_-{\downvdash}\\
}$$
where $\P^A_1$ is the enveloping category of $A$ and $\Fun^{\M}_{/\M}(\P^A_1,\T\M)$ denotes the category of $\M$-enriched lifts
\begin{equation}\label{e:twisted}
\vcenter{\xymatrix@R=1.3pc@C=1.3pc{
& \T\M\ar[d]\\
\P^A_1\ar[r]_{A}\ar@{..>}[ru] & \M
}}
\end{equation}
of the underlying $A$-module $A:\P^A_1 \lrar \M$.
\end{cor}

We note that an enriched functor out of $\P^A_1$ is exactly the notion of an $A$-module. We may hence think of lifts as in~\eqref{e:twisted} as \textbf{twisted modules}. Since the fibers of $\T\M \lrar \M$ are stable these twisted modules are susceptible to the same kind of manipulations as ordinary modules in the stable setting.
 
While Theorem~\ref{t:main-intro} pertains to model categories, it can also be used to obtain results in the $\infty$-categorical setting, using the rectification results of~\cite{PS18} and~\cite{NS}. This is worked out in \S\ref{s:infinity}, where the following $\infty$-categorical analogue of the above result is established (see Theorem~\ref{t:compoo}):
\begin{thm}\label{t:compoo-intro}
Let $\C$ be a closed symmetric monoidal, differentiable presentable $\infty$-category and let $\O^{\otimes} := \rN^{\otimes}(\P)$ be the operadic nerve of a fibrant simplicial operad. Then the forgetful functor induces an equivalence of $\infty$-categories
$$ \T_A\Alg_{\O}(\C) \x{\simeq}{\lrar} \T_{A}\Mod^{\O}_A(\C). $$
\end{thm}
Here $\Mod^{\O}_A(\C)$ is the $\infty$-category of $A$-modules in $\C$, which is closely related to the $\infty$-operad of $A$-modules defined in \cite[\S 3.3]{Lur14} (see Section \ref{s:infinity}). In the special case where $\C$ is stable the conclusion of Theorem~\ref{t:compoo-intro} reduces to the following statement (cf. \cite[Theorem 7.3.4.13]{Lur14}):
\begin{cor}
If, in addition to the above assumptions, $\C$ is stable, then there is an equivalence of $\infty$-categories
$$ \T_A\Alg_{\O}(\C) \x{\simeq}{\lrar} \Mod^{\O}_A(\C) .$$
\end{cor}
While Theorem~\ref{t:compoo-intro} is only applicable to $\infty$-operads which are nerves of simplicial operads (these are most likely all of them, see~\cite{CHH16},\cite{HHM15}), it does cover $\infty$-operads which are not necessarily unital or coherent, as is assumed in~\cite[Theorem 7.3.4.13]{Lur14}. We also note that the model-categorical statement of Theorem~\ref{t:main-intro} can handle not only simplicial operads, but also \textbf{enriched} operads. 
This allows one, for example, to consider dg-operads such as the Lie or Poisson operad, which do not come from simplicial operads and thus are not covered by~\cite[Theorem 7.3.4.13]{Lur14}. It is likely that Theorem~\ref{t:compoo-intro} could be generalized to the setting of \textbf{enriched $\infty$-operads} of \cite{CH17}, see Remark~\ref{r:rectalg}.

One application of the non-stable comparison theorem is that it allows one to study the tangent $\infty$-categories and Quillen cohomology of enriched $\infty$-categories, an application we take up in~\cite{HNP16} (see also Example~\ref{e:enriched-categories-2}). In particular, if $\C$ is an $\infty$-category, then we identify $\T_\C\Cat_{\infty}$ with the $\infty$-category of functors $\Tw(\C) \lrar \Spectra$ from the twisted arrow $\infty$-category of $\C$ to spectra. Similarly, in~\cite{HNP18} we show that if $\C$ is an $(\infty,2)$-category then $\T_\C\Cat_{(\infty,2)}$ can be identified with the $\infty$-category of functors from a suitable \textbf{twisted $2$-cell $\infty$-category} $\Tw_2(\C)$ of $\C$ to spectra.

\subsection*{Conventions and Notation}
There are various points in the text (notably in \S \ref{s:infinity}) where we make use of the theory of $\infty$-categories as developed by Joyal and Lurie. In particular, by an $\infty$-category we will always mean a \textbf{quasicategory} in the sense of Joyal. For a model category $\M$, we will denote by $\M_\infty$ its associated $\infty$-category, obtained as the $\infty$-categorical localization of $\M$ at the weak equivalences. When $\M$ is a combinatorial model category, $\M_\infty$ is a presentable $\infty$-category by \cite[Proposition 1.3.4.22]{Lur14}. By~\cite[Theorem 2.1]{MG16} (see also~\cite[Proposition 1.5.1]{Hin}), any Quillen adjunction $\F: \M \adj \N: \G$ induces an adjunction on the level of $\infty$-categories, which we will denote by 
$$ \F_\infty: \M_\infty \adj \N_\infty: \G_\infty .$$ 

For any model category $\M$ and any small category $\I$, there is a natural functor of $\infty$-categories $\Fun(\I, \M)\lrar \Fun(\I, \M_\infty)$. When $\M$ is combinatorial, this functor realizes $\Fun(\I, \M_\infty)$ as the localization of $\Fun(\I, \M)$ at the pointwise weak equivalences \cite[Proposition 1.3.4.25]{Lur14}, i.e., every diagram in $\M_\infty$ can be rectified to a diagram in $\M$.  By \cite[Proposition 1.3.4.24]{Lur14}, the homotopy colimit of such a rectified diagram presents the colimit in $\M_\infty$.

\subsection*{Acknowledgements}
We would like to thank the anonymous referee, whose many comments and suggestions have greatly improved the paper.
While working on this paper the first author was supported by the Fondation Sciences Math\'ematiques de Paris. The second author was supported by NWO. The third author was supported by Spinoza, ERC 669655 and SFB 1085 grants. 

\section{Tangent categories and tangent bundles}
In this section we will recall the notions of tangent categories and tangent bundles, as well as their model categorical presentations developed in~\cite{part1}. We will then elaborate further on the particular case of (enriched) functor categories (see \S\ref{s:functor}) and establish some results which will be used in \S\ref{s:tangent}.

\subsection{Stabilization of model categories}\label{s:recall}
Recall that a model category is called \textbf{stable} if its homotopy category is pointed and the induced loop-suspension adjunction $\Sig: \Ho(\M) \adj \Ho(\M): \Om$ is an equivalence of categories (equivalently, the underlying $\infty$-category of $\M$ is stable in the sense of~\cite[\S 1]{Lur14}). Given a model category $\M$, it is natural to try to look for a universal stable model category $\M'$ related to $\M$ via a Quillen adjunction $\M \adj \M'$. When $\M$ is combinatorial the underlying $\infty$-category $\M_{\infty}$ is presentable, in which case a universal stable presentable $\infty$-category $\Sp(\M_{\infty})$ admitting a left adjoint functor from $\M_{\infty}$ indeed exists. When $\M$ is furthermore pointed and left proper there are various ways to realize $\Sp(\M_{\infty})$ as a certain model category of spectrum objects in $\M$ (see~\cite{Hov}). 

However, most of these constructions require $\M$ to come equipped with a point-set model for the suspension-loop adjunction (in the form of a Quillen adjunction), which, to our knowledge, is not readily available in some cases of interest, such as enriched categories (see~\cite{HNP16}). As an alternative, the following model category of spectrum objects was developed in \cite{part1}, based on ideas of Heller (\cite{Hel}) and Lurie (\cite{Lur06}): 
\begin{define}\label{d:Om}
Let $\M$ be a weakly pointed combinatorial model category. We will say that a diagram $X_{\bullet\bullet}: \NN\times\NN\lrar \M$ is a \textbf{pre-spectrum} if $X_{n,m}$ is a weak zero object for every $n \neq m$. We will say that it is an \textbf{$\Om$-spectrum} if it is a pre-spectrum and the squares
\begin{equation}\label{e:diag}
\vcenter{\xymatrix@R=1.3pc@C=1.3pc{
X_{n, n}\ar[r]\ar[d] & X_{n, n+1}\ar[d]\\
X_{n+1, n}\ar[r] & X_{n+1, n+1}
}}
\end{equation}
are homotopy Cartesian for every $n \geq 0$. We will say that a map $f: X_{\bullet\bullet} \lrar Y_{\bullet\bullet}$ is a \textbf{stable weak equivalence} if $\Map^{h}(Y_{\bullet\bullet}, Z_{\bullet\bullet})\lrar \Map^{h}(X_{\bullet\bullet}, Z_{\bullet\bullet})$ is a weak equivalence of simplicial sets for every $\Om$-spectrum $Z_{\bullet\bullet}$, where $\Map^{h}$ is the derived mapping space (computed in either the projective or injective model structure on $\NN\times\NN$-diagrams).
\end{define}

\begin{define}\label{d:Om-2}
Let $\M$ be a weakly pointed combinatorial model category. We let $\Sp(\M)$ denote the left Bousfield localization (when it exists) of the projective model structure on the category of $(\NN\times \NN)$-diagrams in $\M$ whose fibrant objects are the levelwise fibrant $\Om$-spectra. The weak equivalences of this model structure are exactly the stable weak equivalences (Definition~\ref{d:Om}). 
\end{define}
The existence of this left Bousfield localization requires some assumptions on $\M$, for example, being combinatorial and left proper. In this case there is a canonical Quillen adjunction
$$ \Sig^{\infty}: \M \adj \Sp(\M): \Om^{\infty}, $$
where $\Om^{\infty}$ sends an $(\NN \times \NN)$-diagram $X_{\bullet\bullet}$ to $X_{0,0}$ and $\Sig^{\infty}$ sends an object $X$ to the constant $(\NN \times \NN)$-diagram with value $X$. While $\Sig^{\infty}X$ may not resemble the classical notion of a suspension spectrum, it can be replaced by one in an essentially unique way, up to a stable weak equivalence (see~\cite[Remark 2.3.4]{part1}). 

When $\M$ is not pointed one stabilizes $\M$ by first forming its \textbf{pointification} $\M_{\ast} := \M_{\ast/}$, endowed with its induced model structure, 
and then forming the above mentioned model category of spectrum objects in $\M_{\ast}$ (see Remark~\ref{r:sound} for when this construction is homotopically sound). We then denote by $\Sig^{\infty}_+: \M \adj \Sp(\M_{\ast}): \Om^{\infty}_+$ the composition of Quillen adjunctions
$$ \xymatrix{
\Sig^{\infty}_+:\M \ar@<1ex>[r]^-{(-) \coprod \ast} & \M_{\ast} \ar@<1ex>[l]^-{\U} \ar@<1ex>[r]^-{\Sig^{\infty}} & \Sp(\M_{\ast}) \ar@<1ex>[l]^-{\Om^{\infty}}: \Om^{\infty}_+ \\
}.$$ 
When $\M$ is a left proper combinatorial model category and $A\in \M$ is an object, the pointification of $\M_{/A}$  
is given by the (left proper combinatorial) model category $\M_{A//A}:=\left(\M_{/A}\right)_{\id_A/}$ of objects in $\M$ over-under $A$, endowed with its induced model structure. The stabilization of $\M_{/A}$ is then formed by taking the model category of spectrum objects in $\M_{A//A}$ as above. 
\begin{define}
Let $\M$ be a combinatorial model category. As in~\cite{part1}, we will denote the resulting stabilization of $\M_{/A}$, when it exists (e.g., when $\M$ is left proper), by
$$ \T_A\M := \Sp(\M_{A//A}) $$
and refer to its as the \textbf{tangent model category} to $\M$ at $A$.
\end{define}

\begin{rem}\label{r:sound}
Even though the slice/coslice categories of a combinatorial model category $\M$ carry an induced model structure, this will in general not be homotopically sound, in the sense that it yields a model for the corresponding slice/coslice $\infty$-categories. However, the slice model category $\M_{/A}$ will be homotopically sound if $A$ is fibrant or if $\M$ is right proper, and the coslice model category $\M_{A/}$ will be homotopically sound if $A$ is cofibrant or $\M$ is left proper (see \cite[Lemma 3.3.1]{HNP18}).
\end{rem}

By~\cite[Proposition 3.3.2]{part1} the $\infty$-category associated to the model category $\T_A\M$ is equivalent to the \textbf{tangent $\infty$-category} $\T_A(\M_\infty)$ in the sense of~\cite[\S 7.3]{Lur14}, as soon as the slice-coslice model category $\M_{A//A}$ is homotopically sound (see Remark~\ref{r:sound}).

Recall that in the $\infty$-categorical setting, the \textbf{tangent bundle} of an $\infty$-category $\C$ is the coCartesian fibration $\T \C \lrar \C$ classified by the functor $\C \lrar \Cat_{\infty}$ sending $A \in \C$ to $\T_A\C$. Starting from a model category $\M$, it is then useful to have an associated model category $\T\M$ whose underlying $\infty$-category is $\T\M_{\infty}$ and which behaves as much as possible like a family of model categories fibered over $\M$, with fibers the various tangent categories of $\M$.

One of the motivations for using the model of~\cite{part1} is that a simple variation of the construction can be used to give a model for the tangent bundle of $\M$ which enjoys the type of favorable formal properties one might expect. For this one considers the following variant of the category $\NN\times\NN$:
\begin{cons}
Let $(\NN \times \NN)_\ast$ be the category obtained from $\NN \times \NN$ by \textbf{freely adding a zero object}. More precisely, the object set of $(\NN \times \NN)_\ast$ is $\Ob(\NN \times \NN) \cup \{\ast\}$, and we have $\Hom_{(\NN \times \NN)_\ast}((n,m),(k,l)) \cong \Hom_{\NN \times \NN}((n,m),(k,l)) \cup \{\ast\}$ for every $(n,m),(k,l) \in \NN \times \NN$, and $\Hom_{(\NN \times \NN)_\ast}(x,\ast) \cong \Hom_{(\NN \times \NN)_\ast}(\ast,x) \cong \{\ast\}$ for every $x \in (\NN \times \NN)_\ast$.

The poset $\NN \times \NN$ has a natural Reedy structure in which the degree of $(n,m)$ is $n+m$ and all morphisms are ascending. The category $(\NN \times \NN)_\ast$ then inherits a natural Reedy structure where the degree of $\ast$ is $0$ and the degree of $(n,m)$ is $n+m+1$. The ascending maps are those which are either in the image of $\NN \times \NN$ or have $\ast$ as their domain, while the descending maps are those which have $\ast$ as their codomain.
\end{cons}

Given a left proper combinatorial model category $\M$, one now defines $\T\M$ as a certain left Bousfield localization of the Reedy model category $\M^{(\NN \times \NN)_{\ast}}_{\Reedy}$, where a Reedy fibrant object $X \in \M^{(\NN \times \NN)_\ast}$ is fibrant in $\T\M$ if and only if the map $X_{n,m} \lrar X_{\ast}$ is a weak equivalence for every $n \neq m$ and for every $n \geq 0$ the square \eqref{e:diag} is homotopy Cartesian. 
We will refer to the model structure on $\T\M$ as the \textbf{tangent bundle} model structure. One may then show that the projection $\T\M \lrar \M$ is both a left and a right Quillen functor and exhibits $\T\M$ as a \textbf{relative model category} over $\M$, in the sense of~\cite{HP}, whose fibers over fibrant objects $A \in \M$ can be identified with the corresponding tangent categories $\T_A\M$. Furthermore, the underlying map of $\infty$-categories $\T\M_\infty \lrar \M_{\infty}$ exhibits $\T\M_{\infty}$ as the tangent bundle of $\M_\infty$.

\subsection{Tangent bundles of functor categories}\label{s:functor}
Let $\M$ be a left proper combinatorial model category tensored over a symmetric monoidal tractable model category $\bS$, and let $\I$ be a small $\bS$-enriched category whose mapping objects are cofibrant. In this case, the \textbf{enriched functor category} $\Fun^{\bS}(\I,\M)$ carries the associated projective model structure. Our goal in this section is to describe the tangent categories and tangent bundle of $\Fun^{\bS}(\I,\M)$. This will be useful in describing the stabilization of module categories in \S\ref{s:tangent}.

By~\cite[Corollary 3.2.2]{part1} the model category $\T\M$ inherits a natural $\bS$-enrichment, and we may hence consider the category $\Fun^{\bS}(\I,\T\M)$ of $\bS$-enriched functors $\I \lrar \T\M$. Since $\M$ is combinatorial, $\T\M$ is combinatorial as well and we may consequently endow $\Fun^{\bS}(\I,\T\M)$ with the projective model structure.
We then have the following proposition:
\begin{pro}\label{c:tangent-functor}
The natural equivalence of categories $\Fun^{\bS}(\I,\M)^{(\NN \times \NN)_{\ast}} \simeq \Fun^{\bS}(\I,\M^{(\NN \times \NN)_{\ast}})$ identifies the model structures
\begin{equation}\label{e:tangent}
\T(\Fun^{\bS}(\I,\M)^{\proj}) \simeq \Fun^{\bS}(\I,\T\M)^{\proj}.
\end{equation}
\end{pro}
\begin{proof}
The tangent bundle model structure on the left hand side of~\eqref{e:tangent} is a left Bousfield localization of the Reedy-over-projective model structure on $\Fun^{\bS}(\I,\M)^{(\NN \times \NN)_{\ast}}$. Similarly, the right hand side is a left Bousfield localization of the projective-over-Reedy model structure on $\Fun^{\bS}(\I,\M^{(\NN \times \NN)_{\ast}})$. It is not hard to verify that the equivalence $\Fun^{\bS}(\I,\M)^{(\NN \times \NN)_{\ast}} \simeq \Fun^{\bS}(\I,\M^{(\NN \times \NN)_{\ast}})$ identifies the Reedy-over-projective model structure with the projective-over-Reedy model structure.
Under this identification the two left Bousfield localizations coincide. Indeed, a levelwise fibrant enriched functor $\F: \I \otimes (\NN \times \NN)_\ast \lrar \M$ 
(where $\otimes$ is inherited from the tensoring of $\bS$ over sets)
is local in either the left or the right hand side of~\eqref{e:tangent} if and only if for every $i \in \I$ the restriction $\F|_{{i} \times (\NN \times \NN)}$ is an $\Om$-spectrum object of $\M_{\F(i,\ast)//\F(i,\ast)}$ (see Definition~\ref{d:Om}).
\end{proof}

\begin{rem}\label{r:global-tensored}
Let $\F: \I \lrar \M$ be a projectively fibrant $\bS$-enriched functor. Since 
$\T\Fun^{\bS}(\I,\M) \lrar \Fun^{\bS}(\I,\M)$
is a relative model category (see~\cite{part1}) the fiber $(\T\Fun^{\bS}(\I,\M))_\F$ inherits a model structure, which coincides in this case with $\T_\F\Fun^{\bS}(\I,\M)$. Because \eqref{e:tangent} is an equivalence of (co)Cartesian fibrations over $\Fun^{\bS}(\I,\M)$, we obtain an equivalence of categories
\begin{equation}\label{e:equiv-functor}
\T_\F\Fun^{\bS}(\I,\M) \simeq \Sp(\Fun^{\bS}(\I,\M)_{\F//\F}) \x{\simeq}{\lrar} \Fun^{\bS}_{/\M}(\I,\T\M) \simeq (\Fun^{\bS}(\I,\T\M))_{\F},
\end{equation}
where $\Fun^{\bS}_{/\M}(\I,\T\M)$ denotes the category of $\bS$-enriched lifts
$$ \xymatrix@R=1.3pc@C=1.3pc{
& \T\M \ar^{\pi}[d] \\
\I \ar@{-->}[ur] \ar_-{\F}[r] & \M. \\
}$$
By transport of structure one obtains a model structure on $\Fun^{\bS}_{/\M}(\I,\T\M)$, which coincides in this case with the corresponding projective model structure (i.e., where weak equivalences and fibrations are defined objectwise). 
\end{rem}

When $\M$ is furthermore stable the situation becomes even simpler. Indeed, in this case $\Fun^{\bS}(\I,\M)$ is stable and is Quillen equivalent to both sides of~\eqref{e:equiv-functor} under mild assumptions. This follows from~\cite[Corollary 3.3.3]{part1} and the following lemma:
\begin{lem}\label{l:ker-stable}
Let $\M$ be a stable model category equipped with a strict zero object $0 \in \M$ and let $A \in \M$ be an object. Assume that either $A$ is cofibrant or $\M$ is left proper and that either $A$ is fibrant or $\M$ is right proper. Then the Quillen adjunction
\begin{equation}\label{e:ker}
(-) \coprod A: \M \adj \M_{A//A}: \ker
\end{equation}
is a Quillen equivalence. In particular, in this case $\M_{A//A}$ is stable.
\end{lem}
\begin{proof}
The functor $\ker$ sends an object $A \lrar C \x{p}{\lrar} A$ over-under $A$ to the object $\ker(p) \cong C \times_A 0$, while its left adjoint sends an object $B$ to the object $A \lrar B \coprod A \lrar A$, where the first map is the inclusion of the second factor and the second map restricts to the identity on $A$ and to the $0$-map on $B$. We then see that~\eqref{e:ker} is indeed a Quillen adjunction.

Let $B \in \M$ be a cofibrant object and $A \lrar C \x{p}{\lrar} A$ a fibrant object of $\M_{A//A}$. We have to show that a map $f:B \coprod A\lrar C$ over-under $A$ is a weak equivalence if and only if the adjoint map $f^{\ad}:B\lrar C \times_A 0$ is a weak equivalence. These two maps fit into a diagram in $\M$ of the form
$$\xymatrix@R=1.3pc@C=1.3pc{
0\ar[d]\ar[r] & B\ar^{f^{\ad}}[r]\ar[d] & C \times_A 0\ar[r]\ar[d] & 0\ar[d]\\
A\ar[r] & B \coprod A \ar^{f}[r] & C\ar^{p}[r] & A,
}$$
where the left square is coCartesian and the right square is Cartesian. Under the assumption that $A$ is cofibrant or $\M$ is left proper the left square is homotopy coCartesian. Under the assumption that $A$ is fibrant or $\M$ is right proper the right square is homotopy Cartesian. Since the external rectangle is clearly homotopy Cartesian and coCartesian and since $\M$ is stable, it follows from~\cite[Remark 2.1.4]{part1} and the pasting lemma for homotopy (co)Cartesian squares that all squares in this diagram are homotopy Cartesian and coCartesian. This means in particular that the top middle horizontal map is a weak equivalence iff the bottom middle horizontal map is one.
\end{proof}

\begin{cor}\label{c:stable-cotangent}
Let $\M$ be a proper combinatorial strictly pointed stable model category. Then the right Quillen functors
\begin{equation}\label{e:tangent-stable}
\Sp(\M_{A//A}) \x{\Om^{\infty}}{\lrar} \M_{A//A} \x{\ker}{\lrar} \M
\end{equation}
are both right Quillen equivalences, and for every cofibrant object $B \lrar A$ in $\M_{/A}$ the image of $\Sig^{\infty}_+(B)$ under the composed Quillen equivalence ~\eqref{e:tangent-stable} is naturally equivalent to $B$ itself.
\end{cor}

\begin{cor}\label{c:stable-functor-tangent}
Let $\M$ be a proper combinatorial strictly pointed stable model category tensored over a tractable SM model category $\bS$. Then for every $\bS$-enriched functor $\F: \I \lrar \M$ the tangent model category $\T_\F\Fun^{\bS}(\I,\M)$ is Quillen equivalent to $\Fun^{\bS}(\I,\M)$.
\end{cor}

\section{Colored operads}
In this section we will recall the notion of a colored symmetric operad and review some of its basic properties. The main technical tool we will need is a suitable natural filtration on free algebras (see \S\ref{ss:filtration}) which plays a key role in the proof of the independence of stabilization on operations of arity $\geq 2$, discussed in \S\ref{s:main}. While this filtration has been studied before by several authors, for our purposes we need a somewhat more specific formulation in which the filtration is directly associated to a certain skeletal filtration on the operad in question. 

\subsection{Preliminaries}\label{ss:prelim}
Throughout this section, let $\M$ be a symmetric monoidal (SM) locally presentable category in which the tensor product distributes over colimits. 

\begin{defn}
Let $\Sigma$ be the groupoid of finite sets and bijections between them. We will use the term $\uline{n}$ to denote a generic set of size $n$.  In particular, the automorphism group $\Aut(\uline{n})$ can be identified with the symmetric group on $n$ elements. For every $\uline{n}$ we will denote by $\uline{n}_+ := \uline{n} \coprod \{\ast\}$. 
\end{defn}

We consider the association $\uline{n} \mapsto \uline{n}_+$ as a functor $(-)_+: \Sig \lrar \Set$. For a set $W$ we will denote by $\Sig_W := \Sig \times_{\Set} \Set_{/W}$ the comma category associated to $(-)_+$. More explicitly, we may identify objects of $\Sig_W$ with pairs $(\uline{n},\ovl{w})$ where $\uline{n}$ is an object of $\Sig$ and $\ovl{w}: \uline{n}_+ \lrar W$ is a map of sets. We think of $\ovl{w}$ as a vector of elements of $W$ indexed by $\uline{n}_+$ 
and refer to $n$ as the \textbf{arity} of the object $\ovl{w}$. We will generally abuse notation and refer to the object $(\uline{n},\ovl{w})$ simply by $\ovl{w}$, suppressing the explicit reference to the arity. 

\begin{rem}\label{r:autw}
The category $\Sig_W$ is a groupoid. 
If $\ovl{w} \in \Sig_W$ has arity $n$ then the automorphism group $\Aut(\ovl{w})$ of $\ovl{w}$ in $\Sig_W$ can be identified with the subgroup of $\Aut(\uline{n})$ consisting of those permutations $\sig:\uline{n} \lrar \uline{n}$ such that $\ovl{w} \circ \sig = \ovl{w}$.

Note that $\Sig_W$ is equivalent to the groupoid whose objects are tuples of elements $(w_1, \dots, w_n, w_\ast)$ in $W$ with $n\geq 0$, where every element of the symmetric group $\sigma\in \Sigma_n$ defines a morphism $\sigma: (w_{\sigma(1)}, \dots, w_{\sigma(n)}, w_\ast)\lrar (w_1, \dots , w_n, w_\ast)$.
\end{rem}

\begin{define}
A \textbf{$W$-symmetric sequence in $\M$} is a functor $X:\Sig_W \lrar \M$. 
We will denote by $\SymSeq_W(\M)$ the category of $W$-symmetric sequences in $\M$. When the category $\M$ is fixed we will often abuse notation and denote $\SymSeq_W(\M)$ simply by $\SymSeq_W$.
\end{define}
The category $\SymSeq_W(\M)$ admits a (non-symmetric) monoidal product known as the \textbf{composition product}. Let us recall the details.

\begin{cons}\label{c:dec}
Let $\text{Ar}$ be the groupoid whose objects are (not necessarily bijective) arrows of finite sets $\phi: \uline{k}\lrar \uline{n}$ and whose morphisms are natural bijections between such maps. 
We will denote by $\Dec_W := \text{Ar}\times_{\Set} \Set_{/W}$ the comma category associated to the functor $\text{sum}_+: \text{Ar}\lrar \Set$ which sends $\phi: \uline{k}\lrar \uline{n}$ to $(\uline{k} \coprod \uline{n})_+$. Explicitly, the objects of $\Dec_W$ are given by tuples $(\phi, \ovl{v})$ consisting of a map of finite sets $\phi: \uline{k}\lrar \uline{n}$ and a map $\ovl{v}: (\uline{k} \coprod \uline{n})_+ \lrar W$. 
\end{cons}
\begin{rem}
An object of $(\uline{n}, \ovl{w})\in\Sig_W$ can be thought of as encoding the domain $(w_1, \dots, w_n)$ and codomain $w_\ast$ of a potential $n$-ary operation. Similarly, we think of an object $(\phi:\uline{k}\lrar \uline{n}, \ovl{v})$ of $\Dec_W$ as describing a potential \textbf{decomposition} of a $k$-ary operation, in the following sense: the restriction $\ovl{v}|_{\uline{k}_+}$ describes a $k$-ary operation $(v_j)_{j\in \uline{k}}\lrar v_\ast$, which comes with a decomposition into a collection of operations $(v_j)_{j\in \phi^{-1}(i)}\lrar v_i$ for each $i\in \uline{n}$, followed by an $n$-ary operation $(v_i)_{i\in \uline{n}}\lrar v_\ast$.
\end{rem}
Given an object $(\phi:\uline{k}\lrar \uline{n}, \ovl{v}) \in \Dec_W$, we will denote by $\phi_*: \uline{k} \coprod \uline{n} \lrar \uline{n}$ the map which restricts to $\phi$ on $\uline{k}$ and to the identity on $\uline{n}$. For every $i \in \uline{n}$ we will consider the inverse image $\phi_*^{-1}(i) \cong \phi^{-1}(i) \cup \{i\}$ as a pointed set with base point $i$. 
We may then consider $(\phi_*^{-1}(i),\ovl{v}|_{\phi_*^{-1}(i)})$ as an object of $\Sig_W$ (which we will just refer to as $\ovl{v}|_{\phi_*^{-1}(i)}$, following our convention above).

\begin{define}\label{d:comp}
For two $W$-symmetric sequences $X$ and $Y$ we define $X\boxtimes Y: \Dec_W\lrar \M$ by
$$ (X\boxtimes Y)\Big(\phi: \uline{k}\lrar \uline{n}, \ovl{v}\Big) = X(\ovl{v}|_{\uline{n}_+})\otimes \Big(\displaystyle\mathop{\bigotimes}_{i \in \uline{n}} Y\Big(\ovl{v}|_{\phi_*^{-1}(i)}\Big)\Big). $$
We then define the \textbf{composition product} of $X$ and $Y$ to be the left Kan extension
$
X\circ Y := \pi_!(X\boxtimes Y)
$
of $X \boxtimes Y$ along the functor $\pi: \Dec_W \lrar \Sig_W$ given by $\pi(\phi,\ovl{v}) = \ovl{v}|_{\uline{k}_+}$. 
\end{define}

Explicitly, for $\ovl{w}$ of arity $k$, the composition product is given by the formula
\begin{equation}\label{e:composition}
(X \circ Y)(\ovl{w}) \cong \coprod_{[(\phi,\ovl{v})]} \Big[X(\ovl{v}|_{\uline{n}_+}) \otimes \Big(\displaystyle\mathop{\bigotimes}_{i \in \uline{n}} Y\Big(\ovl{v}|_{\phi_*^{-1}(i)}\Big)\Big) \Big]\otimes_{\Aut(\phi,\ovl{v})} \Aut(\ovl{w}), 
\end{equation}
where the coproduct runs over all isomorphism classes of objects $(\phi: \uline{k}\lrar \uline{n},\ovl{v}: (\uline{k} \coprod \uline{n})_+ \lrar W) \in \Dec_W$ such that $\ovl{v}|_{\uline{k}_+} = \ovl{w}$, while $\Aut(\phi,v)$ is the automorphism group $(\phi,v)$ in $\Dec_W$. We refer the reader to~\cite[\S 3]{PS18} for more details on the composition product (which is called the ``substitution product'' in loc. cit.).

\begin{rem}\label{r:semi-linear}
The composition product determines a (non-symmetric) monoidal structure on $\SymSeq_W$ whose unit is the symmetric sequence $1_{\SymSeq_W}$ such that $1_{\SymSeq}(\ovl{w})$ is the unit of $\M$ when $\ovl{w}$ is constant and of arity $1$ and is initial otherwise. We note that the composition product preserves colimits in the left entry (see~\cite[Proposition 3.6]{PS18}), but generally not in the right. However, as can be seen by inspecting the formula in~\eqref{e:composition}, for a fixed $X \in \SymSeq_W$ which is concentrated in arity $1$, the functor $X \circ (-)$ does preserve colimits. 
\end{rem}

\begin{define} 
A \textbf{$W$-colored (symmetric) operad} $\P$ is a monoid object in $\SymSeq_W(\M)$ with respect to the composition product described above. We will usually not mention the term ``symmetric'' explicitly when discussing such operads, and will omit the term ``$W$-colored'' whenever $W$ is clear in the context. We will denote by $\Op_W(\M)$ the category of $W$-colored operads in $\M$.
\end{define}
Explicitly, a $W$-colored operad $\P$ consists of objects 
$$
\P(\ovl{w})=\P\big((w_i)_{i\in \uline{n}}; w_\ast\big)
$$
that parametrize the \textbf{$n$-ary operations from $(w_i)_{i \in \uline{n}}$ to $w_{\ast}$}. For every $\phi:\uline{k} \lrar \uline{n}$ and $\ovl{v}: (\uline{k} \coprod \uline{n})_+ \lrar W$ as above, there is a composition operation
$$ \P\big((v_i)_{i\in \uline{n}}; v_\ast\big)\otimes \Big(\displaystyle\mathop{\bigotimes}_{i \in \uline{n}} \P\Big((v_j)_{j\in \phi^{-1}(i)}; v_i\Big)\Big) \lrar \P\big((v_j)_{j\in \uline{k}}; v_\ast\big) ,$$
subject to the natural equivariance, associativity and unitality conditions.

\begin{define}\label{d:mod}
Let $\P$ be a $W$-colored operad in $\M$. A left (resp. right) \textbf{module} over $\P$ is a $W$-symmetric sequence in $\M$ which is a left (resp. right) module over $\P$ with respect to the composition product above. A \textbf{$\P$-algebra} is a left $\P$-module $A \in \SymSeq_W(\M)$ which is concentrated in arity $0$, i.e., such that $A(\ovl{w}) \cong \emptyset_\M$ whenever $\ovl{w}$ is of arity $n > 0$.
\end{define}
Explicitly, a $\P$-algebra is given by an object $A\in \M^W$, together with maps
$$ \P(\ovl{w})\otimes A(w_1)\otimes ... \otimes A(w_n) \lrar A(w_{\ast}), $$
subject to the natural equivariance, associativity and unitality constraints. We denote by $\Alg_{\P}(\M)$ the category of $\P$-algebras and algebra maps. When there is no possibility of confusion we will also denote $\Alg_\P(\M)$ simply by $\Alg_{\P}$. 
\begin{examples}\label{e:basic-1} We have the following basic examples of operads in sets. For any SM category $\M$ as above, they can also be interpreted as operads in $\M$ using the canonical tensoring of $\M$ over sets.
\begin{enumerate}[(1)]
\item\label{e:basic-1mcom}
Commutative algebras are algebras over the \textbf{commutative operad} $\Com$. Similarly, there is an operad $\M\Com$ on two colors $W = \{a,m\}$ whose algebras are pairs of a commutative algebra and a module over it. The sets of $n$-ary operations $(w_i)_{i \in \uline{n}} \lrar w_\ast$ is a singleton if either $w_\ast=w_i=a$ for all $i$ or if $w_\ast = m$ and $w_i=m$ for exactly one $i\in\uline{n}$, and is empty otherwise. 

\item\label{e:basic-1cat}
Let $O$ be a set.  
Then there is an $O \times O$-colored operad $\P_{O}$ whose algebras are the categories with object set $O$. Explicitly, $\P_{O}$ is the symmetrization of the non-symmetric operad whose $n$-ary operations are as follows: for $n \geq 1$ the object of $n$-ary operations from $(x_1,y_1),(x_2,y_2),...,(x_n,y_n)$ to $(x_\ast,y_\ast)$ is a singleton if $x_\ast=x_1,y_\ast=y_n$ and $y_i = x_{i+1}$ for $i=1,...,n-1$, and is empty otherwise. For $n=0$ the object of $0$-ary operations into $(x_\ast,y_\ast)$ is a singleton if $x_\ast=y_\ast$ and is empty otherwise. For any SM category $\M$, $\P_O$-algebras in $\M$ are $\M$-enriched categories with object set $O$. In particular, when $O=*$ the operad $\P_O$ is the \textbf{associative operad}.
\end{enumerate}
\end{examples}
\begin{examples}
In additive categories, there a various other types of algebraic structures that arise as algebras over operads in abelian groups, which need not come from operads in sets. 
\begin{enumerate}[(1)]
\item Lie algebras are algebras over the \textbf{Lie operad} $\Lie$. This is the smallest sub-operad $\Lie\subseteq \Ass$ of the associative operad (in abelian groups) that contains the commutator element $\id-(12)\in \ZZ[\Sigma_2]=\Ass(2)$.
\item Poisson algebras (i.e., commutative algebras with a Lie bracket satisfying the Leibniz rule) are algebras over the \textbf{Poisson operad}.
\end{enumerate}
Similarly, $n$-shifted Poisson algebras are algebras over the $n$-shifted Poisson operad in chain complexes.
\end{examples}

Let $j_n: \Sigma_W^n\lrar \Sigma_W$ and $t_n: \Sigma_W^{\leq n}\lrar \Sigma_W$ be the inclusions of the full subgroupoids consisting of objects of arity $n$ and objects of arity $\leq n$, respectively.

\begin{define}\label{d:skeleta}
Let $\P$ be a $W$-colored symmetric sequence in $\M$. We define the \textbf{arity $n$ part} of $\P$ to be the $W$-symmetric sequence $\P_n:= (j_n)_! j_n^*\P$ and the \textbf{$n$-skeleton} of $\P$ to be the $W$-symmetric sequence $\P_{\leq n}:=(t_n)_! t_n^*\P$, where we use $(-)^*$ to indicate restriction and $(-)_!$ to denote left Kan extension.   
When $n=0$, we denote by $\P^+_{0}$ the free $W$-colored operad generated from the $W$-symmetric sequence $\P_{0} \cong \P_{\leq 0}$. We will say that $\P$ is \textbf{concentrated in arity $n$} if the natural map $\P_n \lrar \P$ is an isomorphism.
\end{define}

Explicitly, the symmetric sequence $\P_n$ (resp.\ $\P_{\leq n}$) is given by $\P_{n}(\ovl{w}) \cong \P(\ovl{w})$ for $\ovl{w}$ of arity $n$ (resp.\ arity $\leq n$) and $\P_{n}(\ovl{w}) \cong \emptyset$ for $\ovl{w}$ of arity $\neq n$ (resp.\ arity $>n$). The operad $\P^+_{0}$ has no non-trivial $m$-ary operations for $m >1$ (i.e., the corresponding objects of $m$-ary operations are all initial), while $\P^+_{0}(\ovl{w}) \cong \P_{0}(\ovl{w})$ for $\ovl{w}$ of arity $0$ and its $1$-ary operations are only identity maps. 

\begin{rem}\label{r:filtration-converge}
The skeleton inclusion maps $\P_{\leq n} \lrar \P$ assemble into a map
$$ \colim_n \P_{\leq n} \lrar \P .$$
This map is an isomorphism: indeed, for every $\ovl{w} \in \Sig_W$ the filtration $\{\P_{\leq n}(\ovl{w})\}_{n \in \NN}$ stabilizes after finitely many steps.
\end{rem}

\begin{rem}\label{r:bimodule}
Let $\P$ be a $W$-colored operad. Then $\P_{\leq 1}$ and $\P_1$ inherit from $\P$ a natural operad structure. Furthermore, $\P_n$ inherits from $\P$ the structure of a $\P_1$-bimodule and $\P_{\leq n}$ inherits from $\P$ the structure of a $\P_{\leq 1}$-bimodule. Similarly, $\P_0 \cong \P_{\leq 0}$ inherits from $\P$ the structure of a $\P$-bimodule, and is in particular a $\P$-algebra. As such, it is the \textbf{initial $\P$-algebra}. 
\end{rem}
\begin{rem}\label{e:p1}
A $W$-colored operad in $\M$ with only 1-ary operations is precisely an $\M$-enriched category with $W$ as its set of objects.  
Consequently, if $\P$ is an operad in $\M$ then we will often consider $\P_1$ as an $\M$-enriched category, and will refer to it as the \textbf{underlying category} of $\P$. When $\P$ is an $\M$-enriched category (i.e., when $\P \cong \P_1$), a $\P$-algebra is simply an enriched functor $\P\lrar \M$.
\end{rem}
\begin{define}
An \textbf{augmented $\P$-algebra} in $\M$ is a $\P$-algebra $A$ equipped with a map of $\P$-algebras $A \lrar \P_0$, where $\P_0$ is considered as the initial $\P$-algebra. We will denote by $\Alg_{\P}^{\aug} := (\Alg_{\P})_{/\P_0}$ the category of augmented $\P$-algebras. We note that by construction the category $\Alg_{\P}^{\aug}$ is pointed. 
\end{define}

Every morphism of $W$-colored operads $f:\P\lrar \Q$ induces an extension-restriction adjunction 
$$\xymatrix{f_!:\Alg_{\P}\ar@<1ex>[r] & \Alg_{\Q}:f^*\ar[l]<1ex>_{\upvdash}.}$$
Let $\int_{\P\in\Op_W}\Alg_{\P}\lrar \Op_W$ be the Grothendieck construction of the functor $\P\mapsto \Alg_{\P}$ and $f\mapsto f_!$. As in~\cite[Definition 1.5]{BM09}, one may consider the section 
$$\Op_W\lrar \int_{\P\in\Op_W}\Alg_{\P},$$ 
sending a $W$-colored operad $\P$ to the pair $(\P,\P_0)$ consisting of $\P$ and its initial $\P$-algebra. This functor admits a left adjoint 
$$\Env:\int_{\P\in\Op_W}\Alg_{\P}\lrar \Op_W, $$
associating to a pair $(\P,A)$ of an operad $\P$ and a $\P$-algebra $A$ a new operad $\P^A := \Env(\P,A) \in \Op_W$. Following~\cite{BM09} we will refer to $\P^A$ as the \textbf{enveloping operad} of $A$, and refer to the $\M$-enriched category $\P^A_1$ as the \textbf{enveloping category} of $A$. The category of algebras over $\P^A$ is equivalent to the category $(\Alg_{\P})_{A/}$ of $\P$-algebras under $A$ (see~\cite[Proposition 4.4(iv)]{PS18}). When $A\cong\P_0$ is the initial $\P$-algebra the natural map $\P\lrar \P^A$ is an isomorphism (\cite[Proposition 4.4(i)]{PS18}).

\begin{defn}\label{d:module}
Let $\P$ be an operad and $A$ a $\P$-algebra. An \textbf{$A$-module} is an algebra over $\P^A_1$, i.e., an $\M$-enriched functor from the enveloping category of $A$ to $\M$. We will denote the category of $A$-modules in $\M$ by $\Mod_A^{\P}(\M)$, or simply by $\Mod^{\P}_A$ when there is no possibility of confusion.
\end{defn}
Unwinding the definition, one finds that a module over a $\P$-algebra $A$ is given by an object $M \in \M^W$ together with action maps
$$\xymatrix{
\P(\ovl{w})\otimes \Big(\displaystyle\mathop{\bigotimes}_{i \in \uline{n} \setminus \{j\}} A(w_i)\Big)\otimes M(w_j)\ar[r] & M(w_{\ast}) 
}$$ 
for every $j \in \uline{n}$, subject to natural equivariance, associativity and unitality conditions (cf. \cite[Definition 1.1]{BM09} for the 1-colored case).

\begin{example}\label{e:basic-2}
For a commutative algebra the notion of an operadic module coincides with that of an ordinary module, while for associative algebras we obtain the notion of a \textbf{bimodule}.
If $\M$ is additive and $A$ is an algebra over the Lie operad, then $\P^A_1$ is the usual enveloping algebra of a Lie algebra; an operadic $A$-module is then a Lie-theoretic module in the classical sense.
\end{example}

\begin{rem}\label{r:p01}
If $\P$ is an operad concentrated in arity $\leq 1$ then $\P$ is naturally isomorphic to the enveloping operad $(\P_1)^{\P_0}$ of $\P_0$ as a $\P_1$-algebra. Considering $\P_1$ an an $\M$-enriched category and $\P_0: \P_1\lrar \M$ as an enriched functor we may then identify $\Alg_{\P_{\leq 1}}$ with the coslice category $\Fun(\P_1,\M)_{\P_0/}$. For example, if $A$ is a $\P$-algebra then the category of $\P^A_{\leq 1}$-algebras is equivalent to the category of $\P_1^A$-algebras under $\P_0^A$, i.e.~$A$-modules $M$ equipped with a map of $A$-modules $A\lrar M$. Similarly, the operad $\P^{A,+}_{\leq 0} := (\P^A)^+_{\leq 0}$ is the operad whose algebras are objects $V\in \M^W$ equipped with a map $A\lrar V$ in $\M^W$.
\end{rem}
\begin{example}\label{e:enriched-categories}
Consider the operad $\P_O$ from Example~\ref{e:basic-1}(\ref*{e:basic-1cat}), whose algebras in $\M$ are $\M$-enriched categories with object set $O$. 
If $\C$ is such an enriched category, then the enveloping operad of $\C$ is an operad $\P^\C:=\P_O^\C$ in $\M$ whose algebras are $\M$-enriched categories with object set $O$ equipped with a map from $\C$. Note that $\P^\C$ does not come from an operad in sets in general. 

Explicitly, $\P^\C$ can be identified with the symmetrization of the non-symmetric operad whose object of $n$-ary operations from $(x_1,y_1),(x_2,y_2),...,(x_n,y_n)$ to $(x_\ast,y_\ast)$ is the given by 
\begin{gather}
\P^\C\big((x_1,y_1),(x_2,y_2),...,(x_n,y_n); (x_\ast,y_\ast)\big) = \nonumber\\
\Map_\C(x_\ast,x_1) \otimes \Map_\C(y_1,x_2) \otimes ... \otimes \Map_\C(y_{n-1},x_n) \otimes \Map_\C(y_n,y_\ast).\label{d:pc}
\end{gather}
The nullary operations are given by $\P^{\C}(\emptyset ; (x_\ast, y_\ast))=\C(x_\ast, y_\ast)$ and the composition of operations uses the compositions in $\P_O$ and in $\C$. In particular, we have the following:
\begin{enumerate}[-]\setlength{\itemsep}{4pt}
 \item The operad $(\P^\C)_{1}$ of unary operations corresponds (see Remark \ref{r:bimodule}) to the $\M$-enriched category with object set $W = O \times O$ and morphism objects $(\P^\C)_{1}((x_1,y_1);(x_\ast,y_\ast)) =  \Map_\C(x_\ast,x_1) \otimes \Map_\C(y_1,y_\ast)$. In other words, it is the $\M$-enriched category $\C^{\op} \otimes \C$. In particular, an operadic module over $\C$ is the same as a $\C$-bimodule, i.e., a functor $\C^{\op} \otimes \C \lrar \M$.
 \item $(\P^\C)_0$ is given by $(\P^\C)_0(x, y)=\Map_{\C}(x, y)$. It is an algebra over $\P^\C$ by inserting $\Map_{\C}(x_i, y_i)$ at each intermediate step in \eqref{d:pc}. In particular, it has the structure of a $\C$-bimodule $\Map_{\C}(-, -): \C^{\op}\otimes\C\lrar \M$.
 \item By Remark \ref{r:bimodule}, the $W$-colored symmetric sequence of $n$-ary operations $(\P^\C)_n$ carries two actions of $(\P^\C)_1=\C^{\op}\otimes\C$: the action $(\P^\C)_1\circ (\P^\C)_n\lrar (\P^\C)_n$ is given on \eqref{d:pc} by precomposition with arrows from $\C$ in $x_\ast$ and postcomposition in $y_\ast$. The action $(\P^\C)_n\circ (\P^\C)_1\lrar (\P^\C)_n$ is given on \eqref{d:pc} by postcomposition in $x_1, \dots, x_n$ and precomposition in $y_1, \dots y_n$.
\end{enumerate}
\end{example}

\subsection{The filtration on a free algebra}\label{ss:filtration}
In this section we will recall the natural filtration on the free algebra over a colored operad $\P$ generated by an object $X$ together with a map $\P_0\lrar X$, i.e.~ the free $\P$-algebra where the nullary operations have already been specified. This is a special case of the filtration on a pushout of $\P$-algebras along a free map $\P \circ Y\lrar \P \circ X$ (see, e.g.,~\cite{PS18},\cite{BM09} and \cite{Cav14}) in the case where $Y = \P_0$ and the pushout is taken along $\P \circ \P_0 \lrar \P_0$. For our purposes we need a somewhat more specific formulation of these results, in which the filtration is directly associated to the filtration of $\P$ by skeleta (see Definition~\ref{d:skeleta}).

Let $\M$ be a closed symmetric monoidal category and let $\P$ be a $W$-colored symmetric sequence in $\M$. 
We will denote by $\O := \P^+_{\leq 0}$ the operad freely generated from $\P_{\leq 0}$ (see Definition~\ref{d:skeleta}). We now recall that $\P_n$ inherits from $\P$ the structure of a $\P_1$-bimodule and $\P_{\leq n}$ inherits from $\P$ the structure of a $\P_{\leq 1}$-bimodule (see Remark~\ref{r:bimodule}). In particular, there is a canonical map $\P_n\lrar \P_{\leq n}$ of left $\P_1$-modules, which induces a map $\P_n\circ \O\lrar \P_{\leq n}$ of $\P_1-\O$-bimodules. 
\begin{lem}\label{l:Q_n}
Let $\P$ be a $W$-colored operad in $\M$. Then for every $n \geq 2$ the square of $\P_{1}-\O$-bimodules
\begin{equation}\label{e:pushout}
\vcenter{\xymatrix@R=1.3pc@C=1.3pc{
\Big(\P_n\circ \O\Big)_{\leq n-1} \ar[r]\ar[d] & \P_n\circ \O \ar[d] \\ 
\P_{\leq n-1} \ar[r] & \P_{\leq n} \\
}}
\end{equation}
is a pushout square. Here the left vertical map is obtained by applying the functor $(-)_{\leq n-1}$ to the right vertical map.
\end{lem}
\begin{proof}
The composition product $(X, Y)\mapsto X\circ Y$ preserves colimits in the first argument, and colimits in the second argument if $X$ is concentrated in arity $1$ (see Remark~\ref{r:semi-linear}).
This implies that the forgetful functor from $\P_{1}-\O$-bimodules to $W$-symmetric sequences preserves and detects colimits, and so it suffices to show that the above square is a pushout square in the category of $W$-symmetric sequences. Since all objects are trivial in arities $>n$ and both horizontal maps are isomorphisms in arities $<n$, it remains to prove that the square in arity $n$ is a pushout square. Indeed, in arity $n$ the left vertical map is an isomorphism between initial objects and the right vertical map is an isomorphism because $\O$ coincides with the unit of $\SymSeq_W$ (with respect to $\circ$) in arities $\geq 1$. 
\end{proof}

Let us now consider the natural operad maps $\O \x{\psi}{\lrar} \P_{\leq 1} \x{\vphi}{\lrar} \P$. The inclusion $\rho = \vphi \circ \psi:\O\lrar \P$ induces a free functor $\rho_!: \Alg_\O\lrar \Alg_\P$. When $X$ is an $\O$-algebra (i.e. an object of $\M^W$ equipped with a map from $\P_0$), $\rho_!(X)$ is given by the relative composition product $\P\circ_{\O} X$ (which, as a $W$-colored symmetric sequence, is concentrated in arity $0$). Lemma~\ref{l:Q_n} together with Remark~\ref{r:filtration-converge} and Remark~\ref{r:semi-linear} then imply the following:
\begin{cor}\label{c:filtration}
The underlying left $\P_{\leq 1}$-module of the free $\P$-algebra $\P \circ_{\O} X$ can be written as a colimit
$ \P \circ_{\O} X \cong \displaystyle\mathop{\colim}_{n \geq 1} \P_{\leq n} \circ_{\O} X 
$,
where for $n \geq 2$ the $n$'th step can be understood in terms of a pushout square of left $\P_1$-modules
\begin{equation}\label{e:pushout-x}
\vcenter{\xymatrix@R=1.3pc@C=1.3pc{
\Big(\P_n\circ \O\Big)_{\leq n-1} \circ_{\O} X \ar[r]\ar[d] & \P_n\circ X \ar[d] \\ 
\P_{\leq n-1} \circ_{\O} X \ar[r] & \P_{\leq n} \circ_{\O} X. \\
}}
\end{equation}
\end{cor}
\begin{example}\label{e:free-cat-1}
Let $\C$ be an $\M$-enriched category with object set $S$ and let $\P:=\P^\C$ be $\M$-enriched operad from Example \ref{e:enriched-categories}. In this case, an $\O$-algebra $X$ is given by a collection of maps $\iota: \Map_\C(x, y)\lrar X(x, y)$ in $\M$. The arrows of the $\M$-enriched category $\P^\C\circ_{\O} X$ are freely generated by arrows from $\Map_\C$ and $X$, subject to the condition that arrows from $\Map_\C$, as well as their images under $\iota: \Map_\C\lrar X$, are composed according to the composition in $\C$. 

The $0$-th stage of the filtration is simply $\C$ itself and the stages $\P^\C_{\leq n} \circ_{\O} X$ contain only sequences of arrows $x_\ast\stackrel{f_1}{\lrar}x_1\stackrel{g_1}{\lrar}y_1\stackrel{f_2}{\lrar}\dots \stackrel{g_k}{\lrar}y_k\stackrel{f_{k+1}}{\lrar}y_\ast$ where the $f_i$ come from $\Map_\C$, the $g_i$ from $X$ and where $k\leq n$ (cf.\ Equation \eqref{d:pc}). In each stage, one adds all sequences containing exactly $n$ arrows from $X$. These are glued along those sequences with at least one arrow in the image of $\iota: \Map_\C\lrar X$; the composition is then already contained in $\P^\C_{\leq n-1} \circ_{\O} X$.
\end{example}

The filtration of Corollary~\ref{c:filtration} is somewhat non-satisfactory: while $\P \circ_{\O} X \cong \colim_n \P_{\leq n} \circ_{\O} X$ is a filtration of $\P \circ_{\O} X$ as a left $\P_{\leq 1}$-module (or, equivalently, as a $\P_{\leq 1}$-algebra, since $\P \circ_{\O} X$ is concentrated in arity $0$), the consecutive steps~\eqref{e:pushout-x} are only pushout squares of left $\P_1$-modules. We note that the difference between the two notions is not big. Since $\P_{\leq 1} \cong \P_1^{\P_0}$ (see~Remark~\ref{r:p01}) we see that if we consider $\P_0$ as a left $\P_{\leq 1}$-module then the category 
of left $\P_{\leq 1}$-modules is naturally equivalent to the category of left $\P_1$-modules under $\P_0$. We may hence fix the situation by performing a mild ``cobase change''.

\begin{define}\label{d:R}
Let $X$ be an $\O$-algebra. We define the map $R^-_n(X) \lrar R^+_n(X)$ by forming the following pushout diagram in the category of left $\P_1$-modules
\begin{equation}\label{e:def-R} 
\vcenter{\xymatrix@R=1.5pc@C=1.3pc{
(\P_n \circ \O)_{0} \ar[r]\ar[d] & (\P_n \circ \O)_{\leq n-1} \circ_{\O} X \ar[d]\ar[r] & \P_n \circ X \ar[d] \\
\P_0 \ar[r] & \poc R^-_n(X) \ar[r] & \poc R^+_n(X).  \\
}}
\end{equation}

As $R^-_n(X)$ and $R^+_n(X)$ are left $\P_1$-modules which carry a map of left $\P_1$-modules from $\P_0$ we may naturally consider both of them as left $\P_{\leq 1}$-modules. We also remark that $R^-_n(X)$ and $R^+_n(X)$ are concentrated in arity $0$ (since all the other objects in the square are), and we may hence consider them also as $\P_{\leq 1}$-algebras.
\end{define}

\begin{lem}\label{l:Q_n-2}
Let $X$ be an $\O$-algebra. Then for every $n \geq 2$ there is a pushout square of $\P_{\leq 1}$-algebras
\begin{equation}\label{e:pushout-R}
\vcenter{\xymatrix@R=1.5pc@C=1.3pc{
R^-_n(X) \ar[r]\ar[d] & R^+_n(X) \ar[d] \\ 
\P_{\leq n-1} \circ_{\O} X \ar[r] & \P_{\leq n} \circ_{\O} X. \poc \\
}}
\end{equation}
\end{lem}
\begin{proof}
We have a commutative diagram of left $\P_1$-modules
\begin{equation}\label{e:rectangle}
\vcenter{\xymatrix@R=1.5pc@C=1.3pc{
(\P_n \circ \O)_{0} \ar[r]\ar[d] & (\P_n \circ \O)_{\leq n-1} \circ_{\O} X \ar[r]\ar[d] & \P_n \circ X \ar[d] \\
\P_0 \ar[r] & \P_{\leq n-1} \ar[r] & \poc \P_{\leq n}. \\
}}
\end{equation}
Using the universal property of pushouts we may extend~\eqref{e:rectangle} to a commutative diagram of the form
\begin{equation}\label{e:multi-pushout} 
\vcenter{\xymatrix@R=1.5pc@C=1.3pc{
(\P_n \circ \O)_{0} \ar[r] \ar[d] & (\P_n \circ \O)_{\leq n-1} \circ_{\O} X \ar[r]\ar[d] &\P_n \circ X \ar[d] \\
\P_0 \ar[r] & \poc R^-_n(X) \ar[r]\ar[d] & \poc R^+_n(X) \ar[d] \ar[d] \\
& \P_{\leq n-1} \ar[r] &  \P_{\leq n}, \\
}}
\end{equation}
where the upper rectangle is the one defining $R^-_n(X) \lrar R^+_n(X)$. 
The right vertical rectangle is just~\eqref{e:pushout-x}, and is hence a pushout rectangle. It then follows that the bottom right square is a pushout square of left $\P_1$-modules, as desired.
\end{proof}
\begin{example}\label{e:free-cat-2}
In the situation of Example \ref{e:free-cat-1}, $\P^\C_n\circ X$ contains finite sequences $\stackrel{f_1}{\lrar}\stackrel{g_1}{\lrar}\dots\stackrel{g_n}{\lrar}\stackrel{f_{n+1}}{\lrar}$ with exactly $n$ arrows from $X$. This object forms a $\C$-bimodule by pre- and postcomposing such sequences with arrows from $\C$.

The object $R^+_n(X)$ is the quotient where we identify sequences for which each $g_i=\iota(h_i)$ is in the image of $\iota: \Map_\C\lrar X$, with the single composed arrow $\iota\big(f_{n+1}h_n\dots h_1f_1\big)$. These composed arrows are precisely the target of a natural map of $\C$-bimodules $\Map_\C\lrar R^+_n(X)$. Note that this map does not lift to a map of bimodules $\Map_\C\lrar \P^\C_n\circ X$.
\end{example}

Our goal is to compute the map of symmetric sequences underlying $R^-_n(X) \lrar R^+_n(X)$. Since this map is the cobase change of the map $(\P_n\circ\O)_{\leq n-1}\circ_\O X\lrar \P_n\circ X$ (Definition \ref{d:R}), it essentially suffices to identify the latter map. To this end, let us introduce the following notation:
\begin{notn}
For a fixed $w_0 \in W$ let us denote by $\Sig_{w_0} \subseteq \Sig_W$ the full subgroupoid spanned by those $\ovl{w}: \uline{k}_+\lrar W$ such that $w_\ast = w_0$.
\end{notn}

\begin{notn}\label{n:sigw0}
For $m \geq 0$, recall the subgroupoids $\Sig^m_W$ and $\Sig^{\leq m}_W$ appearing just above Definition~\ref{d:skeleta}. Given $w_0 \in W$ we will denote by $\Sig^m_{w_0} := \Sig_{w_0} \cap \Sig^m_W$ and $\Sig^{\leq m}_{w_0} := \Sig_{w_0} \cap \Sig^{\leq m}_W$.
\end{notn}
Recall that for any tuple of morphisms $\{f_i: A_i\lrar B_i\}_{i\in \uline{n}}$ in $\M$, the iterated pushout-product map
$$
\Box_{i\in\uline{n}} f_i: Q(\{f_i\})\lrar \bigotimes_{i\in \uline{n}} B_i
$$
is defined as follows. Let $\Sub(\uline{n})$ denote the poset of subsets of $\uline{n}$ and consider the diagram
$$\xymatrix{
\F(\{f_i\}): \Sub(\uline{n})\ar[r] & \M; \hspace{4pt} I\ar@{|->}[r] & \big(\bigotimes_{i\in I} B_i\big)\otimes \big(\bigotimes_{j\in \uline{n}\setminus I} A_j\big).
}$$
If $\Sub^0(\uline{n})\subseteq \Sub(\uline{n})$ is the sub-poset of proper subsets of $\uline{n}$, then the iterated pushout-product map $\Box_{i\in\uline{n}} f_i$ is defined to be the map
\begin{equation}\label{e:pushoutproduct}
Q(\{f_i\}):=\displaystyle\mathop{\colim}_{I\in \Sub^0(\uline{n})}\F(\{f_i\}) \lrar \displaystyle\mathop{\colim}_{I\in \Sub(\uline{n})}\F(\{f_i\})\cong \bigotimes_{i\in\uline{n}} B_i,
\end{equation}
where the isomorphism uses that $\uline{n}$ is a terminal object of $\Sub(\uline{n})$.
\begin{rem}\label{r:pushoutproduct}
For any $i\in \uline{n}$, the map $\Box_{j\in\uline{n}} f_j$ and $\big(\Box_{j\in \uline{n}\setminus i}f_j\big)\Box f_i$ are naturally isomorphic; this can be seen by writing $\Sub^0(\uline{n})$ as a union of the sub-posets $\Sub(\uline{n}\setminus i)$ and $\Sub^0(\uline{n})\setminus\{\uline{n}\setminus i\}$, so that that a colimit over $\Sub^0(\uline{n})$ decomposes as a pushout of colimits over these subposets and their intersection. In particular, $\Box_{j\in\uline{n}} f_j$ can indeed be obtained by inductively taking binary pushout-products.
\end{rem}  
\begin{defn}\label{d:qx}
Let $X$ be an $\O$-algebra, i.e.\ a sequence of objects $\{X(w)\}_{w\in W}$ in $\M$, together with maps $\O(w)\cong \P_0(w)\lrar X(w)$. For every $\ovl{w}\in\Sigma_W$, let us write 
$$
\F_X(\ovl{w}, -):=\F\big(\{\P_0(w_i)\lrar X(w_i)\}_{i\in\uline{n}}\big): \Sub(\uline{n})\lrar \M
$$
and
\begin{equation}\label{e:Qxw}
Q(X, \ovl{w}):=Q\big(\{\P_0(w_i)\rar X(w_i)\}_{i\in\uline{n}})\big) \lrar \bigotimes_{i\in\uline{n}} X(i)
\end{equation}
for the iterated pushout-product map.
\end{defn}
\begin{pro}\label{p:compute-1}
For each $w_0 \in W$, considered as an object in $\Sigma_W$ of arity zero, the map
$$ ((\P_n \circ \O)_{\leq n-1} \circ_{\O}  X)(w_0) \lrar (\P_n \circ X)(w_0) $$
can be identified with the map
\begin{equation}\label{e:Q}
\displaystyle\mathop{\colim}_{\ovl{w} \in \Sig^n_{w_0}} \Big[\P_n(\ovl{w}) \otimes Q(X,\ovl{w})\Big] \lrar\displaystyle\mathop{\colim}_{\ovl{w} \in \Sig^n_{w_0}}\Big[ \P_n(\ovl{w}) \otimes \Big(\bigotimes_{i \in \uline{n}} X(w_i)\Big)\Big] 
\end{equation}
induced by the pushout-product map \eqref{e:Qxw}.
\end{pro}
\begin{cor}\label{c:compute-2}
For any $w_0 \in W$, Diagram~\eqref{e:def-R}
evaluated at $w_0$ is isomorphic to
$$
\resizebox{35em}{!}{
\xymatrix@C=1em@R=2pc{
\displaystyle\mathop{\colim}_{\ovl{w} \in \Sig^n_{w_0}} \Big[\P_n(\ovl{w}) \otimes \Big(\bigotimes_{i \in \uline{n}} \P_0(w_i)\Big)\Big] \ar[r] \ar[d] & \displaystyle\mathop{\colim}_{\ovl{w} \in \Sig^n_{w_0}} \Big[\P_n(\ovl{w}) \otimes Q(X,\ovl{w})\Big] \ar[r]\ar[d] &\displaystyle\mathop{\colim}_{\ovl{w} \in \Sig^n_{w_0}}\Big[ \P_n(\ovl{w}) \otimes \Big(\bigotimes_{i \in \uline{n}} X(w_i)\Big)\Big] \ar[d] \\
\P_0(w_0) \ar[r] & \poc R^-_n(X)(w_0) \ar[r] & \poc R^+_n(X)(w_0).  \\
}}$$
In particular, the map $\vphi_{w_0}:R^-_n(X)(w_0) \lrar R^+_n(X)(w_0)$ is a cobase change of the map 
$$ 
\displaystyle\mathop{\colim}_{\ovl{w} \in \Sig^n_{w_0}} \Big[\P_n(\ovl{w}) \otimes Q(X,\ovl{w})\Big] 
\lrar 
\displaystyle\mathop{\colim}_{\ovl{w} \in \Sig^n_{w_0}}\Big[\P_n(\ovl{w}) \otimes \Big(\bigotimes_{i \in \uline{n}} X(w_i)\Big)\Big].
$$
\end{cor}
The remainder of this section is devoted to the proof of Proposition \ref{p:compute-1}. 

\begin{notn}
Recall the decomposition groupoid $\Dec_W$ of Construction~\ref{c:dec} which is endowed with the functor $\pi: \Dec_W \lrar \Sig_W$ sending $(\uline{k} \lrar \uline{n},\ovl{v})$ to $(\uline{k},\ovl{v}|_{\uline{k}_+})$. 
For $m \geq 0$ we will denote by 
$$ \Dec^m_W := \pi^{-1}\Sig^m_W \qquad\text{and}\qquad \Dec^{\leq m}_W := \pi^{-1}\Sig^{\leq m}_W $$ 
the corresponding preimage subgroupoids. Similarly, for a $w_0 \in W$ we will denote by 
$$ \Dec_{w_0} := \pi^{-1}\Sig_{w_0}, \quad\quad \Dec^m_{w_0} := \pi^{-1}\Sig^m_{w_0} \quad\text{and}\quad \Dec^{\leq m}_{w_0} := \pi^{-1}\Sig^{\leq m}_{w_0} $$ 
the corresponding preimage subgroupoids.
\end{notn}

Let us start with some preliminary observations about the composition product:
\begin{lem}\label{l:comp}
Let $X,Y,Z \in \SymSeq_W(\M)$ be symmetric sequences such that $X$ is concentrated in arity $0$. Then the following assertions hold:
\begin{enumerate}[(1)]
\item $Y\circ X$ is concentrated in arity $0$ and for every $w_0\in W$ (considered as a $0$-arity object of $\Sigma_W$) the inclusion $\Sig_{w_0} \cong \Dec^0_{w_0} \subseteq \Dec_W$ induces an isomorphism
$$
(Y \circ X)(w_0) \cong \displaystyle\mathop{\colim}_{[\ovl{w}:\uline{k}_+ \lrar W] \in \Sig_{w_0}}Y(\ovl{w}) \otimes\Big( \displaystyle\mathop{\bigotimes}_{i \in \uline{k}} X(w_i)\Big).
$$
\item 
For $m \geq 0$ we have 
$$
(Y\circ Z)_{\leq m}\cong \pi^{\leq m}_!\big(Y\boxtimes Z\big)|_{\Dec^{\leq m}_W},
$$
where we denote by $\pi^{\leq m}: \Dec^{\leq m}_W \hrar \Dec_W \lrar \Sig_W$ is the composed functor.
\item For every $w_0\in W$, $((Y \circ Z)_{\leq m} \circ X)(w_0)$ is naturally isomorphic to
\begin{equation}
\displaystyle\mathop{\colim}_{\x{(\phi:\uline{k} \lrar \uline{n},\ovl{v}) 
\in}{\Dec^{\leq m}_{w_0}}} \Big[
Y(\ovl{v}|_{\uline{n}_+}) \otimes \Big(\displaystyle\mathop{\bigotimes}_{j \in \uline{n}} Z\Big(\ovl{v}|_{\phi_*^{-1}(j)}\Big)\Big) \otimes \Big(\displaystyle\mathop{\bigotimes}_{i \in \uline{k}} X(v_i)\Big)
\Big].
\end{equation}
\end{enumerate}
\end{lem}
\begin{proof}
For (1), if $X$ is concentrated in arity 0 then $X$ then $Y\boxtimes X: \Dec_W\lrar \M$ takes initial values outside $\Dec^0_W$ and is hence a left Kan extension of $(Y \boxtimes X)|_{\Dec^0_W}$. It then follows that $Y \circ X$ is given by the left Kan extension of $(Y \boxtimes X)|_{\Dec^0_W}$ along the composed functor $\pi^0:\Dec^0_W \lrar \Dec_W \lrar \Sig_W$. For a given $w_0 \in \Sig^0_W \subseteq \Sig_W$ the fiber inclusion $(\pi^0)^{-1}(w_0) \subseteq \Dec^0_W$ identifies with the full inclusion $\Sig_{w_0} \cong\Dec^0_{w_0} \subseteq \Dec^0_W$, hence the result follows. 

Assertion (2) follows immediately from the Beck-Chevalley property of left Kan extensions (see~\cite[Proposition 11.6]{Joy}) associated to the homotopy Cartesian square of groupoids
$$ \xymatrix@R=1.3pc@C=1.3pc{
\Dec^{\leq m}_W \ar[r]\ar[d] & \Dec_W \ar[d] \\
\Sig^{\leq m}_W \ar[r] & \Sig_W. \\
}$$

For (3), we can use (1), (2) and the projection formula (which holds in this case since $\M$ is closed monoidal, see~\cite{FHM03}) to compute
\begin{align*}
((Y \circ Z)_{\leq m} \circ X)(w_0) &\cong \displaystyle\mathop{\colim}_{[\ovl{w}: \uline{k}_+ \lrar W] \in \Sig_{w_0}}\Big[\big(\pi^{\leq m}_!(Y\boxtimes Z)\big)(\ovl{w}) \otimes\Big( \displaystyle\mathop{\bigotimes}_{i \in \uline{k}} X(w_i)\Big)\Big]\\
&\cong \displaystyle\mathop{\colim}_{\x{(\phi:\uline{k} \lrar \uline{n},\ovl{v}) 
\in}{\Dec^{\leq m}_{w_0}}} \Big[\big(Y\boxtimes Z\big)\big(\uline{k}\stackrel{\phi}{\lrar} \uline{n}, \ovl{v}\big)\otimes\Big( \displaystyle\mathop{\bigotimes}_{i \in \uline{k}} X(v_i)\Big)\Big].
\end{align*}
Inserting the formula for $Y\boxtimes Z$ (see Definition \ref{d:comp}) yields the desired result. 
\end{proof}

\begin{proof}[Proof of Proposition \ref{p:compute-1}]
The formula for $(\P_n\circ X)(w_0)$ follows immediately from part (1) of Lemma \ref{l:comp}; indeed, since $\P_n$ is concentrated in arity $n$ the diagram $\ovl{w} \mapsto Y(\ovl{w}) \otimes\Big( \bigotimes_{i \in \uline{k}} X(w_i)\Big)$ takes initial values outside the essential image of $\Sig^n_{w_0} \subseteq \Sig_{w_0}$, and is hence a left Kan extension of its restriction to $\Sig^n_{w_0}$. 

Let us now establish the formula for $(\P_n \circ \O)_{\leq n-1} \circ_{\O} X$. By definition it is given as the coequalizer of the diagram
\begin{equation}\label{e:bar}\vcenter{\xymatrix{
\Big(\P_n\circ \O\Big)_{\leq n-1}\circ\O \circ X \ar@<0.5ex>[r] \ar@<-0.5ex>[r] & 
\Big(\P_n\circ \O\Big)_{\leq n-1} \circ X, \\
}}\end{equation}
where the top arrow is induced from the left $\O$-module structure of $X$ and the bottom arrow from the right $\O$-module structure of $\Big(\P_n\circ \O\Big)_{\leq n-1}$. 

We can use part (3) of Lemma \ref{l:comp} to describe the symmetric sequences in this diagram, which are all concentrated in arity $0$. To describe their value at $w_0$, let us consider the following:
\begin{enumerate}[($\star$)]
\item
Let $\X_{w_0} \subseteq \Dec^{\leq n-1}_{w_0}\subseteq \Dec_W$ be the full subgroupoid spanned by the objects
$(\phi: \uline{k} \lrar \uline{m},\ovl{v}:(\uline{k} \coprod \uline{m})_+ \lrar W)$ with $k \leq n-1$ satisfying the following conditions: $m=n$ (i.e., $\uline{m}$ has $n$ elements), $\phi$ is injective and $v: (\uline{k} \coprod \uline{m})_+ \lrar W$ factors as $(\uline{k} \coprod \uline{m})_+ \lrar \uline{m}_+ \lrar W$ and $v_{\ast} = w_0$. 
\end{enumerate}
Since $\P_n$ is concentrated in arity $n$ and $\O$ contains only identities and $0$-ary operations, the diagram $\Dec^{\leq n-1}_{w_0} \lrar \M$ given by
$$ \big(\uline{k}\stackrel{\phi}{\lrar}\uline{m}, \ovl{v}) \mapsto \P_n(\ovl{v}|_{\uline{m}_+}) \otimes \Big(\bigotimes_{j \in \uline{m}} \O\big(\ovl{v}|_{\phi_*^{-1}(j)}\big)\Big) \otimes \Big(\bigotimes_{i \in \uline{k}} X(v_i)\Big) $$
takes initial values outside (the essential image of) the full subgroupoid $\X_{w_0} \subseteq \Dec^{\leq n-1}_{w_0}$ of ($\star$). We may therefore replace the colimit in part (3) of Lemma \ref{l:comp} by a colimit over $\X_{w_0}$ to obtain
$$
((\P_n \circ \O)_{\leq n-1} \circ X)(w_0) \cong 
\displaystyle\mathop{\colim}_{\x{(\phi:\uline{k} \lrar \uline{n},\ovl{v}) 
\in}{\X_{w_0}}}
\Big[
\P_n(\ovl{v}|_{\uline{n}_+}) \otimes \Big(\displaystyle\mathop{\bigotimes}_{j \in \uline{n}} \O\Big(\ovl{v}|_{\phi_*^{-1}(i)}\Big)\Big) \otimes \Big(\displaystyle\mathop{\bigotimes}_{i \in \uline{k}} X(v_i)\Big)
\Big].
$$
Consider the functor
$$ q:\X_{w_0} \lrar \Sig^n_{w_0}; \quad\quad (\phi:\uline{k} \lrar \uline{n},v) \mapsto (\uline{n},v|_{\uline{n}_+}) .$$
The functor $q$ is fibered in sets, where the fiber over $\ovl{w} \in \Sig^n_{w_0}$ is equivalent to the set of proper subsets $I \subsetneq \uline{n}$. An automorphism of $\ovl{w}$ acts on this fiber via the inclusion $\Aut(\ovl{w}) \subseteq \Aut(\uline{n})$ (see Remark \ref{r:autw}). Breaking the colimit along the fibration $q:\X_{w_0} \lrar \Sig^n_{w_0}$ of ($\star$), we find that
\begingroup\makeatletter\def\f@size{9}\check@mathfonts\makeatother
\begin{flalign}\label{f:computation}
((\P_n \circ \O)_{\leq n-1} \circ X)(w_0) & \cong \displaystyle\mathop{\colim}_{\ovl{w} \in \Sig^n_{w_0}} \displaystyle\mathop{\coprod}_{I \subsetneq \uline{n}} 
\Big[\P_n(\ovl{w}) \otimes\Big(\displaystyle\mathop{\bigotimes}_{\uline{n} \setminus I} \O(w_j)\Big) \otimes \Big(\displaystyle\mathop{\bigotimes}_{i \in \uline{k}} X(w_i)\Big)
\Big]\\
\nonumber &\cong \displaystyle\mathop{\colim}_{\ovl{w} \in \Sig^n_{w_0}}\bigg[\P_n(\ovl{w}) \otimes\coprod_{I \subsetneq \uline{n}}\Big[ \Big(\bigotimes_{j\in \uline{n} \setminus I}\P_0(w_j)\Big)\otimes \Big(\bigotimes_{i\in I} X(w_i)\Big)\Big]\bigg].
\end{flalign}\endgroup
\noindent The second isomorphism is due to the commutativity of tensor products with coproducts in each variable separately and the fact that the $0$-ary operations of $\O$ are those of $\P$. 

Similarly, note that $\O\circ X$ is the free $\O$-algebra on $X$, which is naturally isomorphic to $\P_0\coprod X$. We may therefore compute
\begin{flalign}\label{f:computation-2}
 &((\P_n \circ \O)_{\leq n-1} \circ \O \circ  X)(w_0) \cong &\\ 
\nonumber & \displaystyle\mathop{\colim}_{\ovl{w} \in \Sig^n_{w_0}}\bigg[\P_n(\ovl{w}) \otimes\coprod_{I \subsetneq \uline{n}} \Big[\Big(\bigotimes_{j\in \uline{n} \setminus I}\P_0(w_j)\Big)\otimes \Big(\bigotimes_{i\in I} \big(\P_0(w_i) \coprod X(w_i)\big)\Big)\Big]\bigg] \cong &\\
\nonumber &\displaystyle\mathop{\colim}_{\ovl{w} \in \Sig^n_{w_0}}
\bigg[\P_n(\ovl{w}) \otimes\coprod_{I' \subseteq I \subsetneq \uline{n}}\Big[ \Big(\bigotimes_{j\in \uline{n} \setminus I}\P_0(w_j)\Big)\otimes\Big( 
\bigotimes_{i \in I \setminus I'} \P_0(w_i) \Big)\otimes \Big(\bigotimes_{i' \in I'} X(w_{i'})\Big)
\Big]\bigg] \cong &\\
\nonumber &\displaystyle\mathop{\colim}_{\ovl{w} \in \Sig^n_{w_0}}
\bigg[\P_n(\ovl{w}) \otimes\coprod_{I' \subseteq I \subsetneq \uline{n}} \Big[\Big(\bigotimes_{j\in \uline{n} \setminus I'}\P_0(w_j)\Big)\otimes \Big(\bigotimes_{i' \in I'} X(w_{i'})\Big)
\Big]\bigg],
\end{flalign}
where the second isomorphism is due to the commutativity of tensor products with coproducts in each variable separately. Combining~\eqref{f:computation} and~\eqref{f:computation-2} we may now identify $(\P_n \circ \O)_{\leq n-1} \circ_{\O} X)(w_0)$ with the coequalizer
$$ \xymatrix@R=0em{
\displaystyle\mathop{\colim}_{\ovl{w} \in \Sig^n_{w_0}}
\Big[\P_n(\ovl{w}) \otimes\coprod_{I' \subseteq I \subsetneq \uline{n}} \Big[\Big(\bigotimes_{j\in \uline{n} \setminus I'}\P_0(w_j)\Big)\otimes \Big(\bigotimes_{i' \in I'} X(w_{i'})\Big)
\Big]\Big]  \ar@<0.5ex>[r] \ar@<-0.5ex>[r] & \\
\displaystyle\mathop{\colim}_{\ovl{w} \in \Sig^n_{w_0}}\Big[\P_n(\ovl{w}) \otimes\coprod_{I \subsetneq \uline{n}}\Big[ \Big(\bigotimes_{j\in \uline{n} \setminus I}\P_0(w_j)\Big)\otimes \Big(\bigotimes_{i\in I} X(w_i)\Big)\Big]\Big].  &\\
}$$
Using now the notation $\F_X(\ovl{w},I) := \Big(\bigotimes_{j\in \uline{n} \setminus I}\P_0(w_j)\Big)\otimes \Big(\bigotimes_{i\in I} X(w_i)\Big)$ of Definition \ref{d:qx}) and the commutativity of tensor products with colimits in each variable separately we may write the coequalizer above as
\begin{equation}\label{e:computation-3} 
\displaystyle\mathop{\colim}_{\ovl{w} \in \Sig^n_{w_0}}\bigg[\P_n(\ovl{w})\otimes\coeq\Big[
\xymatrix{
\displaystyle\mathop{\coprod}_{I' \subseteq I \subsetneq \uline{n}}\F_X(\ovl{w},I') \ar@<0.5ex>[r] \ar@<-0.5ex>[r] & \displaystyle\mathop{\coprod}_{I \subsetneq \uline{n}}\F_X(\ovl{w},I)\\
}\Big]\bigg].
\end{equation}
Here the bottom map sends the component $\F_X(\ovl{w},I')$ to the same component on the right hand side, while the top map sends it to the component $\F_X(\ovl{w},I)$ using the structure maps $\P_0(w_i) \lrar X(w_i)$ for $i \in I \setminus I'$. We now observe that the coequalizer appearing in the inner brackets of~\eqref{e:computation-3} is exactly the degree $\leq 1$ part of the bar construction computing the colimit of $\F$ on $\Sub^0(\uline{n})$ (see~\cite[\S V.2]{Mac13}). We may finally conclude that
$$ ((\P_n \circ \O)_{\leq n-1} \circ_{\O} X)(w_0) \cong \displaystyle\mathop{\colim}_{\ovl{w} \in \Sig^n_W(w_0)} \P_n(\ovl{w}) \otimes Q(X,\ovl{w})$$
and that
$((\P_n \circ \O)_{\leq n-1} \circ_{\O} X)(w_0) \lrar (\P_n \circ X)(w_0)$
identifies with the map
$$ \displaystyle\mathop{\colim}_{\ovl{w} \in \Sig^n_W(w_0)}\Big[ \P_n(\ovl{w}) \otimes Q(X,\ovl{w})\Big] \lrar \displaystyle\mathop{\colim}_{\ovl{w} \in \Sig^n_W(w_0)}\Big[ \P_n(\ovl{w}) \otimes \Big(\bigotimes_{i \in \uline{n}} X(w_i)\Big)\Big], $$
as desired.
\end{proof}

\section{Stabilization of algebras over operads}\label{ss:comparison}
In this section we will establish the main results of this paper, following the outline given in the introduction. In \S\ref{s:main}, we will establish that stabilization is insensitive to algebraic operations of arity $\geq 2$, by giving an equivalence between the stabilization of the category of augmented algebras over an operad $\P$ and the stabilization of the category of algebras over its $1$-skeleton $\P_{\leq 1}$.  

We then show in \S\ref{s:tangent} how this comparison result can be used to equate tangent categories of algebras with tangent categories of modules. The latter can then be described explicitly as suitable categories of enriched lifts, using \S\ref{s:functor}. In the last section \S\ref{s:infinity} we show how to harness the results of \S\ref{s:main} to obtain analogous results in the $\infty$-categorical setting.  

Throughout this section, we will assume that $\M$ is a \textbf{combinatorial SM model category} which is \textbf{differentiable} in the following sense (cf.\ \cite[Definition 6.1.1.6]{Lur14}): for every homotopy finite category $\I$ (i.e., a category whose nerve is a finite simplicial set), the right derived limit functor $\RR\lim: \M^{\I}\lrar \M$ preserves $\NN$-indexed homotopy colimits. A Quillen pair $\F: \M\adj \N: \G$ is differentiable if $\M$ and $\N$ are differentiable and $\RR\G$ preserves $\NN$-indexed homotopy colimits.

Recall that an operad $\P$ is called \textbf{$\Sig$-cofibrant} if the underlying symmetric sequence of $\P$ is projectively cofibrant, and \textbf{admissible} if the model structure on $\M$ transfers to the category $\Alg_\P$ of $\P$-algebras.
When $\P$ is admissible we will also consider the category $\Alg_{\P}^{\aug}$ of augmented algebras as a model category with its slice model structure. 

\begin{defn}
We will say that $\P$ is \textbf{stably admissible} if it is admissible and in addition the stable model structure on $\Sp(\Alg_{\P}^{\aug})$ exists.
\end{defn}

\begin{rem}\label{r:admissible}
One case where stable admissibility can often be established is when $\P$ is $1$-skeletal, i.e., $\P \cong \P_{\leq 1}$. Indeed, recall from Remark \ref{r:p01} that a $1$-skeletal operad $\P$ is simply an $\M$-enriched category $\P_1$ together with an enriched functor $\P_0: \P_1\lrar \M$. The category of $\P$-algebras is then equivalent to the category $\Fun(\P_1, \M)_{\P_0/}$ of enriched functors $\P_1\lrar \M$ under $\P_0$. When $\P$ is $\Sig$-cofibrant we can endow $\Fun(\P_1, \M)_{\P_0/}$ with the coslice model structure associated to the projective model structure on $\Fun(\P_1,\M)$. Under the equivalence of categories $\Alg_\P \simeq \Fun(\P_1,\M)_{\P_0/}$ this model structure is the one transferred from $\M^W$. In particular, any $1$-skeletal $\Sig$-cofibrant operad in $\M$ is admissible. Furthermore, in this case the forgetful functor $\Alg_\P^{\aug} \lrar \M^W_{\P_0//\P_0}$ is a left Quillen functor. It then follows that $\Alg_\P^{\aug}$ is left proper (and hence that $\P$ is stably admissible) as soon as $\P$ is $\Sig$-cofibrant and $\M$ is left proper.
\end{rem}

\subsection{The comparison theorem}\label{s:main}

Our goal in this section is to prove the core result of this paper, which relates the stabilization of $\Alg_{\P}^{\aug}$ to the stabilization of the simpler category $\Alg_{\P_{\leq 1}}^{\aug}$, obtained by forgetting the operations of arity $\geq 2$. First recall that the map $\vphi: \P_{\leq 1} \lrar \P$ induces an adjunction 
$$\xymatrix{
\vphi^{\aug}_!:\Alg_{\P_{\leq 1}}^{\aug}\ar@<1ex>[r] & \Alg_{\P}^{\aug} \ar@<1ex>[l]_-{\upvdash}}:\vphi^*_{\aug}
$$
on augmented algebras and hence an adjunction
$$\xymatrix{
\vphi_{!}^{\Sp} := \Sp(\vphi^{\aug}_!):\Sp(\Alg_{\P_{\leq 1}}^{\aug})\ar@<1ex>[r] & \Sp(\Alg_{\P}^{\aug}) \ar@<1ex>[l]_-{\upvdash}}:\Sp(\vphi^*_{\aug}) =: \vphi^*_{\Sp}
$$
on spectrum objects. Our main theorem can then be formulated as follows:

\begin{thm}\label{t:comp}
Let $\M$ be a differentiable, left proper, combinatorial SM model category and let $\P$ be a $\Sig$-cofibrant stably admissible operad in $\M$. Assume either that $\M$ is right proper or that $\P_0$ is fibrant. Then the induced Quillen adjunction
$$ \xymatrix{
\vphi^{\Sp}_!:\Sp(\Alg_{\P_{\leq 1}}^{\aug}) \ar@<1ex>[r] & \Sp(\Alg_{\P}^{\aug}) \ar@<1ex>[l]_-{\upvdash}}: \vphi^*_{\Sp} 
$$
is a Quillen equivalence.
\end{thm}

\begin{rem}\label{r:fresse}
If every object in $\M$ is cofibrant and $\P$ is a cofibrant single-colored operad (with respect to the transferred model structure on operads in $\M$) then $\P$ is admissible (\cite[Theorem 4]{Spi01}) and the association $A \mapsto \P^A$ preserves weak equivalences (\cite[Theorem 17.4.B(b)]{fresse}). This implies that $\Alg_{\P}$ (as well as $\Alg_{\P}^{\aug}$ ) is left proper and hence that $\P$ is stably admissible. Work of Rezk (\cite{rezk}) gives the same conclusion for a colored cofibrant operad when $\M$ is the category of simplicial sets. It seems very likely that this statement holds for every cofibrant colored operad and every combinatorial model category $\M$ in which every object is cofibrant.
\end{rem}

\begin{example}\label{e:enriched-categories-2}
Suppose that $\M$ is as in Theorem \ref{t:comp} and that $\C$ is a fibrant $\M$-enriched category with set of objects $O$. Recall the $\M$-enriched operad $\P^\C$ of Example \ref{e:enriched-categories}, whose algebras are $\M$-enriched categories with object set $O$, equipped with a functor from $\C$. Algebras over $\P^\C_{\leq 1}$ are bimodules $\F: \C^{\op}\otimes\C\lrar \M$  equipped with a bimodule map $\Map_\C\lrar \F$.

The operad $\P^\C$ is $\Sigma$-cofibrant, since it is induced from a non-symmetric operad. When $\M$ is sufficiently nice (e.g., if every object is cofibrant and weak equivalences are stable under filtered colimits \cite[Remark 3.1.7]{HNP16}, \cite[Theorem A.3.2.4]{Lur09}), $\P^\C$ is also stably admissible. For example, one could take $\M$ to be simplicial sets (with the Joyal or the Kan-Quillen model structure) or chain complexes over a field. In this case, Theorem \ref{t:comp} identifies the stabilization of $\M$-enriched categories over-under $\C$ with object set $O$ with the stabilization of bimodules over-under $\Map_\C$. In \cite{HNP16}, we will develop this idea further in order to study the tangent categories and Quillen cohomology of enriched categories where the object set is not fixed.
\end{example}

The key ingredient in the proof of Theorem~\ref{t:comp} is the following statement: 
\begin{pro}\label{p:comp}
Let $\M$ be a differentiable, combinatorial SM model category and let $\P$ be an admissible $\Sig$-cofibrant operad in $\M$ (which implies the same for $\P_{\leq 1}$, see Remark~\ref{r:admissible}). Assume either that $\M$ is right proper or that $\P_0$ is fibrant. Consider the induced adjunction on $\NN\times\NN$-diagrams
$$\xymatrix{
\vphi_!^{\NN\times\NN}: (\Alg_{\P_{\leq 1}}^{\aug})^{\NN\times\NN}\ar@<1ex>[r] & (\Alg_{\P}^{\aug})^{\NN\times\NN} \ar@<1ex>[l]_-{\upvdash}: \vphi^*_{\NN \times \NN}}.
$$
Then the following two statements hold:
\begin{enumerate}[(1)]
\item the functor $\vphi^*_{\NN \times \NN}$ preserves and detects stable weak equivalences between pre-spectra (see Definition~\ref{d:Om}).
\item for any levelwise cofibrant pre-spectrum object $X_{\bullet\bullet}$ in $\Alg_{\P_{\leq 1}}^{\aug}$, the unit map $u_X: X_{\bullet\bullet}\lrar \vphi^*_{\NN \times \NN}\vphi_!^{\NN \times \NN}X_{\bullet\bullet}$ is a stable weak equivalence.
\end{enumerate}
\end{pro}

\begin{proof}[Proof of Theorem~\ref{t:comp} assuming Proposition~\ref{p:comp}]
Since every $\Om$-spectrum is in particular a pre-spectrum we see that every object is stably equivalent to a levelwise cofibrant pre-spectrum object. Combining (1) and (2) of Proposition~\ref{p:comp} we may then deduce that the derived unit of the Quillen adjunction $\vphi^{\Sp}_! \dashv \vphi^*_{\Sp}$ is a natural equivalence. On the other hand, since $\vphi^*_{\aug}$ preserves and detects weak equivalences it follows that $\vphi^*_{\Sp}$ detects stable weak equivalences between $\Om$-spectra, and hence induces a conservative functor on the level of $\infty$-categories. 
This means that the adjunction $\vphi^{\Sp}_! \dashv \vphi^*_{\Sp}$ induces an equivalence on the level of $\infty$-categories and is hence a Quillen equivalence.
\end{proof}

In fact, Proposition~\ref{p:comp} also yields an analogue of Theorem~\ref{t:comp} when $\P$ is not assumed to be stably admissible (see Corollary~\ref{c:compmodel}). 
For this let us recall some notation from~\cite{part1}. 
\begin{define}\label{d:sp'}
For a weakly pointed, differentiable, combinatorial model category $\N$, let $\Sp'(\N) \subseteq \N^{\NN \times \NN}$ denote the full subcategory spanned by $\Om$-spectra, considered as a \textbf{relative category} with respect to levelwise weak equivalences, and $\Sp''(\M) \subseteq \N^{\NN \times \NN}$ the full subcategory spanned by pre-spectra, considered as a relative category with respect to \textbf{stable weak equivalences}.
\end{define} 

We observe that the inclusion $\Sp'(\N) \subseteq \Sp''(\N)$ is an equivalence of relative categories. 
This follows from the fact that one can functorially replace a levelwise cofibrant pre-spectrum $X$ by an $\Om$-spectrum $X^{\Om}$ equipped with a stable weak equivalence $X \lrar X^{\Om}$ (see~\cite[Remark 2.1.9]{part1}) and the fact that a stable weak equivalence between $\Om$-spectra is a levelwise weak equivalence. Of course, when the stable model structure exists this is just a direct corollary of the fact that every object in $\Sp(\M)$ is stably equivalent to a pre-spectrum (see~\cite[Remark 2.3.6]{part1}).
It now follows from~\cite[Remarks 3.3.4]{part1} that the underlying $\infty$-categories of both $\Sp'(\N)$ and $\Sp''(\N)$ model the $\infty$-categorical stabilization $\Sp(\N_{\infty})$.

\begin{cor}\label{c:compmodel}
Let $\M$ be a differentiable, combinatorial SM model category and let $\P$ be an admissible $\Sig$-cofibrant operad in $\M$ (which implies the same for $\P_{\leq 1}$, see Remark~\ref{r:admissible}). Assume either that $\M$ is right proper or that $\P_0$ is fibrant. Then 
$\vphi^*_{\aug}$ induces an equivalence of relative categories
$$ \Sp'(\Alg_\P^{\aug}) \lrar \Sp'(\Alg_{\P_{\leq 1}}^{\aug}). $$
In particular, $\Sp(\Alg_{\P_{\leq 1}}^{\aug})$ is a model for the stabilization of $(\Alg_\P^{\aug})_{\infty}$.
\end{cor}
\begin{proof}
Let $Q: \big(\Alg_{\P_{\leq 1}}^{\aug}\big){}^{\NN \times \NN} \lrar \big(\Alg_{\P_{\leq 1}}^{\aug}\big){}^{\NN \times \NN}$ be a levelwise cofibrant replacement functor.
Since the functor $\vphi^*_{\NN \times \NN}: \big(\Alg_{\P}^{\aug}\big){}^{\NN \times \NN} \lrar \big(\Alg_{\P_{\leq 1}}^{\aug}\big){}^{\NN \times \NN}$ preserves $\Om$-spectra, it follows that $\vphi^{\NN \times \NN} _!\circ Q$ preserves stable weak equivalences between pre-spectra. By Proposition~\ref{p:comp}(1), $\vphi^*_{\NN \times \NN}$ also preserves stable weak equivalences between pre-spectra. It follows that $\vphi^{\NN \times \NN} _!\circ Q$ and $\vphi^*_{\NN \times \NN}$ induce relative functors between $\Sp''(\Alg_\P^{\aug})$ and $\Sp''(\Alg_{\P_{\leq 1}}^{\aug})$. Combining (1) and (2) of Proposition~\ref{p:comp}, we conclude that the compositions $\vphi^*_{\NN \times \NN} \circ \vphi^{\NN \times \NN} _!\circ Q$ and $\vphi^{\NN \times \NN} _!\circ Q \circ \vphi^*_{\NN \times \NN}$ are both related to the identity functors by chains of natural weak equivalences. In particular, $\vphi^*_{\NN \times \NN}: \Sp''(\Alg_{\P}^{\aug}) \lrar \Sp''(\Alg_{\P_{\leq 1}}^{\aug})$ is an equivalence of relative categories and hence $\vphi^*_{\NN \times \NN}:\Sp'(\Alg_{\P}^{\aug}) \lrar \Sp'(\Alg_{\P_{\leq 1}}^{\aug})$ is an equivalence of relative categories as well.
\end{proof}

The rest of this section is devoted to the proof of Proposition~\ref{p:comp}. We begin with some preliminary lemmas. We will say that a map $f: X \lrar Y$ in a weakly pointed model category $\M$ is \textbf{null-homotopic} if its image in $\Ho(\M)$ factors through the zero object $0$. 

\begin{define}
Let $\M_1,...,\M_n,\N$ be weakly pointed model categories and let $\F: \prod_i \M_i \lrar \N$ be a functor (of ordinary categories). We will say that $\F$ is \textbf{multi-reduced} if $\F(X_1,...,X_n)$ is a weak zero object of $\N$ whenever all the $X_i$ are cofibrant and at least one of them is a weak zero object.
\end{define}

\begin{lem}[{cf. \cite[Proposition 6.1.3.10]{Lur14}}]\label{l:reduced}
For $n \geq 2$, let $\M_1,...,\M_n$ and $\N$ be combinatorial differentiable weakly pointed model categories 
and let $\F: \prod_i \M_i \lrar \N$ be a multi-reduced functor. For every collection $\{Z^i_{\bullet\bullet} \in \M_i^{\NN\times\NN}\}_{1 \leq i \leq n}$ of levelwise cofibrant pre-spectrum objects, the object $\F(Z^1_{\bullet\bullet},...,Z^n_{\bullet\bullet}) \in \N^{\NN \times \NN}$ is stably equivalent to a weak zero object.
\end{lem}
\begin{proof}
For simplicity we will prove the claim for $n=2$. The proof in the general case is similar. Consider the following commutative diagram
$$ \xymatrix@R=1.4pc@C=1.3pc{
\F(Z^1_{n,n},Z^2_{n,n}) \ar[r]\ar[d] & \F(Z^1_{n,n+1}, Z^2_{n,n}) \ar^{\sim}[r]\ar[d] & \F(Z^1_{n,n+1},Z^2_{n,n+1})\ar^{\sim}[d] \\
\F(Z^1_{n+1,n},Z^2_{n,n}) \ar[r]\ar^{\sim}[d] & \F(Z^1_{n+1,n+1},Z^2_{n,n}) \ar[r]\ar[d] & \F(Z^1_{n+1,n+1},Z^2_{n,n+1}) \ar[d] \\
F(Z^1_{n+1,n},Z^2_{n+1,n}) \ar^{\sim}[r] & \F(Z^1_{n+1,n+1},Z^2_{n+1,n}) \ar[r] & \F(Z^1_{n+1,n+1},Z^2_{n+1,n+1}). \\
}$$
Since $\F$ is multi-reduced we have that $\F(X,Z^2_{k,m})$ and $\F(Z^1_{m,k},X)$ are weak zero objects for every $k \neq m$ and any cofibrant $X \in \M$, so that all off-diagonal items in this diagram are weak zero objects.
The external square induces a map $f_n:\Sig \F(Z^1_{n,n},Z^2_{n,n}) \lrar \F(Z^1_{n+1,n+1},Z^2_{n+1,n+1})$ in the homotopy category $\Ho(\N)$, which factors as
$$
\Sig \F(Z^1_{n,n},Z^2_{n,n})\lrar \F(Z^1_{n+1,n+1},Z^2_{n,n}) \lrar \F(Z^1_{n+1,n+1},Z^2_{n+1,n+1}),
$$
where the first map is induced from the top left square. Since the second map is null-homotopic, it follows that the map $f_n$ is null-homotopic as well. By~\cite[ Corollary 2.4.6]{part1} the $\NN \times \NN$-diagram $\F(Z^1_{\bullet\bullet},Z^2_{\bullet\bullet})$ is stably equivalent to an $\Om$-spectrum whose value at the place $(m,m)$ can be computed as a homotopy colimit of the form
\begingroup\makeatletter\def\f@size{9}\check@mathfonts\makeatother
$$ 
\F(Z^1_{m,m},Z^2_{m,m}) \x{g_m}{\lrar} \Om\F(Z^1_{m+1,m+1}, Z^2_{m+1,m+1}) \x{g_{m+1}}{\lrar} \Om^2\F(Z^1_{m+2,m+2},Z^2_{m+2,m+1}) \rar \dots ,
$$\endgroup
where the image of $g_{i}$ in $\Ho(\N)$ is adjoint to $f_{i}$ and hence null-homotopic for every $i \geq m$. Since a homotopy colimit of a sequence of null-homotopic maps is a weak zero object the desired result follows.
\end{proof}

Recall that for any map $f: X \lrar Y$ in a category with a zero object $0$, the \textbf{cofiber} of $f$, denoted $\cof(f)$, is the pushout $0 \coprod_X Y$. 

\begin{lem}\label{l:cof}
Let $\M$ be a strictly pointed combinatorial model category and suppose that $f: X \lrar Y$ is a levelwise cofibration between levelwise cofibrant pre-spectra in $\M$. Then $f$ is a stable weak equivalence if and only if the map $0 \lrar \cof(f)$ is a stable weak equivalence.
\end{lem}
\begin{proof}
Under our assumptions the pushout $0 \coprod_X Y$ is also a homotopy pushout in $\M^{\NN \times \NN}$. It follows that if $f$ is a stable weak equivalence then $0 \lrar \cof(f)$ is a stable weak equivalence. We shall now show that if $0 \lrar \cof(f)$ is a stable weak equivalence then $f$ is a stable weak equivalence. Consider the diagram
$$
\xymatrix@R=1.4pc@C=1.3pc{
X \ar^{f}[r]\ar[d] & Y \ar[d] \ar[r] & Z_1 \ar[d] \\
0 \ar[r] & \cof(f) \ar[r]\ar[d] &\poc X' \ar[d] \\
& Z_2 \ar[r] & \poc Y', \\
}
$$
in which all the squares are homotopy coCartesian and $Z_1,Z_2$ are weak zero objects. If $0 \lrar \cof(f)$ is a stable weak equivalence then the map $\cof(f) \lrar Z_2$ is a stable weak equivalence and hence the map $X' \lrar Y'$ is a stable weak equivalence. On the other hand, since the external rectangles are homotopy coCartesian it follows that the map $X' \lrar Y'$ is a model for the induced map $\Sig X \lrar \Sig Y$ on suspensions. Now for every $\Om$-spectra $W$ we have $\Map^h(X',W[1]) \simeq \Map^h(X,\Om W[1]) \simeq \Map^h(X,W)$ and the same for $Y'$. Since $X' \lrar Y'$ is a stable weak equivalence it now follows that $f: X \lrar Y$ is a stable weak equivalence. 
\end{proof}

\begin{cor}\label{c:equiv}
Let $\M$ be a combinatorial differentiable SM model category 
and let $A^1,...,A^n \in \M$ be a collection of cofibrant objects (with $n\geq 2$). For each $i=1,...,n$ let
$ A^i \x{f^i_{\bullet\bullet}}{\lrar} X^i_{\bullet\bullet} \lrar A^i $ 
be a levelwise cofibrant pre-spectrum object in $\M_{A^i//A^i}$. Then the levelwise pushout-product
$$ f^1_{\bullet\bullet} \Box ... \Box f^n_{\bullet\bullet}: Q(f^1_{\bullet\bullet},...,f^n_{\bullet,\bullet}) \lrar \bigotimes_{i=1}^{n}X^i_{\bullet\bullet} $$
is a stable weak equivalence and levelwise cofibration between levelwise cofibrant pre-spectrum objects in $\M_{A^1 \otimes ... \otimes A^n//A^1 \otimes ... \otimes A^n}$.
\end{cor}
\begin{proof}
The pushout-product axiom in $\M$ implies that $f^1_{\bullet\bullet} \Box ... \Box f^n_{\bullet\bullet}$ is a levelwise cofibration between levelwise cofibrant objects. By Lemma~\ref{l:cof} it will now suffice to show that the cofiber of this map is stably equivalent to the zero pre-spectrum in $\M_{A^1 \otimes ... \otimes A^n//A^1 \otimes ... \otimes A^n}$.

Consider the functor $\G: \prod_{i=1}^{n}\M_{A^i//A^i} \lrar \M_{A^1 \otimes ... \otimes A^n//A^1 \otimes ... \otimes A^n}$ given by 
$$ \G(A^1 \x{f^1}{\lrar} X^1 \lrar A^1,...,A^n \x{f^n}{\lrar} X^n \lrar A^n) := \cof(f^1 \Box ... \Box f^n).$$
This functor is multi-reduced: indeed, the cofiber $\cof(f^1 \Box ... \Box f^n)$ is a levelwise weak zero object if at least one of the $f^i$ is a trivial cofibration in $\M$, by the pushout-product axiom. Lemma \ref{l:reduced} now implies that the cofiber of the map $f^1_{\bullet\bullet} \Box ... \Box f^n_{\bullet\bullet}$ is stably equivalent to a zero object, as desired.
\end{proof}

Let us now fix a combinatorial SM model category $\M$, a set of colors $W$ and a $W$-colored operad $\P$ in $\M$. We will be interested in the maps of operads $\O := \P^+_{\leq 0} \x{\psi}{\lrar} \P_{\leq 1} \x{\vphi}{\lrar} \P$, whose composition we denote by $\rho$.
Here $\vphi$ is the skeleton inclusion and $\psi$ is induced from $\P_{\leq 0}\lrar \P_{\leq 1}$.  
Upon passing to operadic algebras, this sequence yields a sequence of adjunctions:
\begin{equation}\label{e:operad-seq-2}
\rho_!:\xymatrix{\Alg_\O \ar[r]<1ex>^{\psi_!} & \Alg_{\P_{\leq 1}}\ar[l]<1ex>_{\upvdash}^{\psi^*}\ar[r]<1ex>^{\vphi_!} & \Alg_\P:\rho^*.\ar[l]<1ex>_{\upvdash}^{\vphi^*}}
\end{equation}
The functors $\psi_!,\rho_!$ and $\vphi^*$ preserve initial objects, and consequently augmented objects. 
In particular, the unit map determines a natural map
$$ u: \psi_!(X) \cong \P_{\leq 1}\circ_{\O} X \lrar \P \circ_{\O} X \cong \vphi^*\rho_!(X), $$
of augmented $\P_{\leq 1}$-algebras.

\begin{pro}\label{p:filtration-triv}
Let $X_{\bullet\bullet} \in (\Alg_{\O}^{\aug})^{\NN\times \NN}$ be a levelwise cofibrant pre-spectrum object in augmented $\O$-algebras. If $\P$ is $\Sig$-cofibrant then the induced map
$$ u:\P_{\leq 1} \circ_{\O} X_{\bullet\bullet} \lrar \P \circ_{\O} X_{\bullet\bullet} $$
is a stable weak equivalence in $(\Alg_{\P_{\leq 1}}^{\aug})^{\NN\times\NN}$.
\end{pro}

\begin{proof}
Note that $\P_{\leq 1}\circ_{\O} X_{\bullet\bullet}$ is a levelwise cofibrant pre-spectrum object in $\Alg_{\P_{\leq 1}}^{\aug}$, since $X_{\bullet\bullet}$ is a levelwise cofibrant pre-spectrum object in $\Alg_{\O}^{\aug}$. Similarly, $\P\circ_{\O} X_{\bullet\bullet}$ is a levelwise cofibrant pre-spectrum object in $\Alg_{\P}^{\aug}$, and hence also defines a pre-spectrum object in $\Alg_{\P_{\leq 1}}^{\aug}$. In particular, the map $u$ is a map between pre-spectrum objects.

The forgetful functor $\psi^*_{\aug}:\Alg_{\P_{\leq 1}}^{\aug} \lrar \Alg_{\O}^{\aug}$ is both a left and a right Quillen functor (see Remark~\ref{r:admissible}) and so the Quillen pair $\psi_!^{\aug} \dashv \psi^*_{\aug}$ is in particular differentiable. Since $\psi^*_{\aug}$ preserves and detects weak equivalences, \cite[Corollary 2.4.8]{part1} implies that $\psi^*_{\NN \times \NN}:(\Alg_{\P_{\leq 1}}^{\aug})^{\NN\times\NN} \lrar (\Alg_{\O}^{\aug})^{\NN\times\NN}$ preserves and detects stable weak equivalences between pre-spectra.
It will hence suffice to show that 
$$ \psi^*_{\NN \times \NN}(u):\psi^*_{\NN \times \NN}(\P_{\leq 1} \circ_{\O} X_{\bullet\bullet}) \lrar \psi^*_{\NN \times \NN}(\P \circ_{\O} X_{\bullet\bullet}) $$ 
is a stable weak equivalence. Now by Corollary~\ref{c:filtration} the map
$\psi^*_{\NN \times \NN}(u)$ is a transfinite composition 
\begin{equation}\label{e:psiseq}\psi^*_{\NN \times \NN}(\P_{\leq 1} \circ_{\O} X_{\bullet\bullet}) \lrar \psi^*_{\NN \times \NN}(\P_{\leq 2} \circ_{\O} X_{\bullet\bullet}) \lrar \psi^*_{\NN \times \NN}(\P_{\leq 3} \circ_{\O} X_{\bullet\bullet})\lrar ...
\end{equation}
of maps of $\NN\times\NN$-diagrams of augmented $\O$-algebras.
It suffices to prove that each map
$$ \psi^*_{\NN \times \NN}(\P_{\leq n-1} \circ_{\O} X_{\bullet\bullet}) \lrar \psi^*_{\NN \times \NN}(\P_{\leq n} \circ_{\O} X_{\bullet\bullet}) $$
is a stable weak equivalence and a levelwise cofibration. 
Indeed, $\psi^*_{\NN \times \NN}(u)$ can then be identified with the canonical map from the levelwise cofibrant object $\psi^*_{\NN \times \NN}(\P_{\leq 1} \circ_{\O} X_{\bullet\bullet})$ to the levelwise homotopy colimit of the sequence of stable weak equivalences \eqref{e:psiseq}. It is hence a stable weak equivalence.

Now since $\psi^*_{\aug}$ is left Quillen, it preserves the pushout square \eqref{e:pushout-R} of Lemma \ref{l:Q_n-2}.
By~\cite[Remark 2.1.10]{part1} it then suffices to show that for every $w_0 \in W$ and every $n \geq 2$ the map
\begin{equation}\label{e:Rn}
\psi^*_{\NN \times \NN}(R^-_n(X_{\bullet\bullet}))(w_0) \lrar \psi^*_{\NN \times \NN}(R^+_n(X_{\bullet\bullet}))(w_0)
\end{equation}
is a stable weak equivalence and a levelwise cofibration between levelwise cofibrant $\NN \times \NN$-diagrams in $\M_{\P_0(w_0)//\P_0(w_0)}$.

Let us now fix a color $w_0 \in W$ and a number $n \geq 1$. As in \S\ref{ss:filtration}, let us denote by $\Sig^n_{w_0} \subseteq \Sig^n_W$ the full subgroupoid spanned by those $\ovl{w} \in \Sig^n_W$ such that $w_\ast = w_0$ (see Notation~\ref{n:sigw0}). 
We have the injectively cofibrant functor $\P_0^{\otimes n}:\Sig^n_{w_0} \lrar \M$ given by $\P_0^{\otimes n}(\ovl{w}):=\bigotimes_{i\in\uline{n}} \P_0(w_i)$. For $k,m \in \NN$ consider the functors $X^{\otimes n}_{k,m},Q_{k,m}: \Sig^n_{w_0} \lrar \M$ with $X_{k,m}^{\otimes n}(\ovl{w}):=\bigotimes_{i\in\uline{n}} X_{k,m}(w_i)$ and such that $Q_{k,m}(\ovl{w})$ is the domain of the pushout-product of the maps $\P_0(w_i) \lrar X_{k,m}(w_i)$ for $i=1,...,n$. Corollary~\ref{c:equiv} now implies that the natural map
\begin{equation}\label{e:box}
Q_{\bullet\bullet} \lrar  X^{\otimes n}_{\bullet\bullet}
\end{equation}
is a stable weak equivalence and a levelwise cofibration between levelwise cofibrant pre-spectrum objects in $(\M^{\Sig^n_{w_0}}_{\inj})_{\P_0^{\otimes n}//\P_0^{\otimes n}}$.

Now recall that the coend operation $\M^{\Sig^n_{w_0}}_{\proj} \times \M^{\Sig^n_{w_0}}_{\inj} \lrar \M$, which we will denote by 
$F,G \mapsto F \otimes_{\Sig^n_{w_0}} G$, 
is a left Quillen bifunctor (see, e.g.,~\cite[Remark A.2.9.27]{Lur09}). Let $\P^{n}_{w_0}: \Sig^n_{w_0} \lrar \M$ be the functor $\ovl{w} \mapsto \P(\ovl{w})$. Since $\P$ is $\Sig$-cofibrant we have that $\P^{n}_{w_0}$ is projectively cofibrant, and so we may consider the left Quillen functor $\L$ given by the composition
$$ \xymatrix{
\L:(\M^{\Sig^n_{w_0}}_{\inj})_{\P_0^{\otimes n}/} \ar^-{\P^{n}_{w_0} \otimes_{\Sig^n_{w_0}} (-)}[rr]& & \M_{\P^{n}_{w_0} \otimes_{\Sig^{n}_{w_0}} \P^{\otimes n}_0/} \ar[r] & \M_{\P_0(w_0)/}, \\
}$$ 
where the second functor is the cobase change along the map $\P^{n}_{w_0} \otimes_{\Sig^{n}_{w_0}} \P^{\otimes n}_0 \lrar \P_0(w_0)$ induced by the $\P$-algebra structure of $\P_0$. Proposition~\ref{p:compute-1} now tells us that the map~\eqref{e:Rn} is obtained by levelwise applying (the augmented version of) $\L$ to the map~\eqref{e:box}, and is hence a stable weak equivalence and a levelwise cofibration, as desired.
\end{proof}

To deduce Proposition~\ref{p:comp} from Proposition~\ref{p:filtration-triv} we will need the following result: 
\begin{pro}\label{p:ps}
Let $\M$ be a differentiable, combinatorial, SM model category and let $f: \P \lrar \Q$ be a map of $\Sig$-cofibrant admissible $W$-colored operads in $\M$ which is an isomorphism on $0$-ary operations. Suppose that either $\M$ is right proper or that $\P_0 \cong \Q_0$ is fibrant. Then $ f^*_{\aug}: \Alg_{\Q}^{\aug} \lrar \Alg_{\P}^{\aug} $
preserves and detects weak equivalences as well as sifted homotopy colimits. Furthermore, the Quillen adjunction $f_!^{\aug} \dashv f^*_{\aug}$ is differentiable and induces a monadic adjunction of $\infty$-categories
\begin{equation}\label{e:monadic} 
\xymatrix{
(f^{\aug}_{!})_{\infty}: (\Alg_{\P}^{\aug})_{\infty}\ar@<1ex>[r] & (\Alg_{\Q}^{\aug})_{\infty} \ar@<1ex>[l]_-{\upvdash} \\
}:(f_{\aug}^{*})_{\infty} .
\end{equation}
\end{pro}
\begin{proof}
Since the model structures on both $\Alg_{\P}$ and $\Alg_{\Q}$ are transferred from $\M^W$ and since $f^*\P_0 \cong \Q_0$ by assumption we see that $f^*_{\aug}$ preserves and detects weak equivalences. By Proposition~\ref{p:sifted} both forgetful functors $\Alg_{\P} \lrar \M^W$ and $\Alg_{\Q} \lrar \M^W$ preserve (and hence also detect, since they detect weak equivalences) sifted homotopy colimits. Since the forgetful functors $\Alg_{\P}^{\aug} \lrar \Alg_\P$ and $\Alg_{\Q}^{\aug} \lrar \Alg_\Q$ are left Quillen functors which detect weak equivalences they preserve and detect homotopy colimits. We may hence conclude that $f^*_{\aug}$ preserves and detects sifted homotopy colimits. By the $\infty$-categorical Barr-Beck theorem (see~\cite[Theorem 4.7.4.5]{Lur14}), and using the fact that $\infty$-categorical sifted colimits can be computed as sifted homotopy colimits (see the Conventions and Notations at the end of \S\ref{s:intro}) we may conclude that~\eqref{e:monadic} is monadic.

Since sequential diagrams are in particular sifted we have by the same argument that sequential homotopy colimits in $\Alg_{\P}^{\aug}$ and $\Alg_{\Q}^{\aug}$ are preserved and detected in $\M$. In addition, since we assume that either $\M$ is right proper or that $\P_0 \cong \Q_0$ is fibrant we have that homotopy pullbacks in $\Alg_{\P}^{\aug}$ and $\Alg_{\Q}^{\aug}$ are preserved and detected in $\M$. Since $\M$ is differentiable it then follows that sequential homotopy colimits in $\Alg_{\P}^{\aug}$ and $\Alg_{\Q}^{\aug}$ commute with homotopy pullbacks. In addition, since the poset $\NN$ is weakly contractible it follows that sequential homotopy colimits always preserve final objects. It then follows that $\Alg_{\P}^{\aug}$ and $\Alg_{\Q}^{\aug}$ are differentiable. Furthermore, since sequential diagrams are in particular sifted we get from the above that $f_!^{\aug} \dashv f^*_{\aug}$ is a differentiable Quillen adjunction, as desired. 
\end{proof}

\begin{proof}[Proof of Proposition~\ref{p:comp}]
By Proposition~\ref{p:ps} the Quillen adjunction 
$$ \vphi^{\aug}_!: \Alg_{\P_{\leq 1}}^{\aug} \adj \Alg_{\P}^{\aug}: \vphi^*_{\aug} $$ 
is differentiable. Since $\vphi^*_{\aug}$ preserves weak equivalences \cite[Corollary 2.4.8]{part1} now implies that $\vphi^{\ast}_{\NN \times \NN}: (\Alg_{\P}^{\aug})^{\NN\times\NN} \lrar (\Alg_{\P_{\leq 1}}^{\aug})^{\NN\times\NN}$ preserves stable weak equivalences between pre-spectra. This proves Claim (1).

To prove (2), note that the collection of pre-spectra and the collection of stable weak equivalences are both closed under levelwise homotopy colimits of $(\NN \times \NN)$-diagrams. Note that $\vphi^{\ast}_{\NN \times \NN}$ preserves levelwise weak equivalences and by Proposition~\ref{p:ps} also sifted levelwise homotopy colimits. This means that the collection of levelwise cofibrant pre-spectra $X_{\bullet\bullet} \in \left(\Alg_{\P_{\leq 1}}^{\aug}\right)^{\NN \times \NN}$ for which the unit 
$$ u_X: X_{\bullet\bullet}\lrar \vphi^*_{\NN \times \NN}\vphi_!^{\NN \times \NN}X_{\bullet\bullet} $$
is a stable weak equivalence is closed under sifted levelwise homotopy colimits in $\left(\Alg_{\P_{\leq 1}}^{\aug}\right)^{\NN \times \NN}$. 

Now Proposition~\ref{p:ps} implies that the free-forgetful adjunction
\begin{equation}\label{e:adj-O-P1}
\xymatrix{
\left(\Alg_{\O}^{\aug}\right)^{\NN \times \NN}_{\infty}\ar@<1ex>[r] & \left(\Alg_{\P_{\leq 1}}^{\aug}\right)^{\NN \times \NN}_{\infty} \ar@<1ex>[l]_-{\upvdash}}
\end{equation}
is a monadic adjunction of $\infty$-categories. We note that both functors in this adjunction preserves pre-spectrum objects. Since the collection of pre-spectrum objects is closed under homotopy colimits it follows that~\eqref{e:adj-O-P1} induces a monadic adjunction on the corresponding full subcategories spanned by pre-spectra. Consequently, every pre-spectrum in $\big(\Alg_{\P_{\leq 1}}^{\aug}\big)^{\NN \times \NN}_\infty$ can be obtained as the $\infty$-categorical colimit of a sifted diagram in the image of the pre-spectra in $\big(\Alg_{\O}^{\aug}\big)^{\NN \times \NN}_{\infty}$. By rectifying these diagrams (see the Conventions and Notations at the end of \S\ref{s:intro}), we obtain the following result at the model-categorical level: every pre-spectrum object of $\big(\Alg_{\P_{\leq 1}}^{\aug}\big)^{\NN \times \NN}$ can be written as a sifted levelwise homotopy colimit of pre-spectra weakly equivalent to $\P_{\leq 1} \circ_{\O} Y_{\bullet\bullet}$ for $Y_{\bullet\bullet}: \NN \times \NN\lrar \Alg_{\O}^{\aug}$ a pre-spectrum object. It will hence suffice to show that 
$$ u_{\P_{\leq 1} \circ_{\O} Y_{\bullet\bullet}} : \P_{\leq 1} \circ_{\O} Y_{\bullet\bullet} \lrar \vphi^*_{\NN \times \NN}\vphi_!^{\NN \times \NN}(\P_{\leq 1} \circ_{\O} Y_{\bullet\bullet}) \cong \P \circ_{\P_{\leq 1}}\P_{\leq 1} \circ_{\O} Y_{\bullet\bullet} \cong \P \circ_{\O} Y_{\bullet\bullet} $$
is a stable weak equivalence for every levelwise cofibrant pre-spectrum $Y_{\bullet\bullet}: \NN\times\NN\lrar \Alg_{\O}^{\aug}$. But this is exactly the content of Proposition~\ref{p:filtration-triv}, and so the proof is complete.
\end{proof}

\subsection{Tangent categories of algebras and modules}\label{s:tangent}
Our goal in this section is to explain how, under suitable conditions, Theorem~\ref{t:comp} can be used to identify the tangent category $\T_A\Alg_{\P}$ at a $\P$-algebra $A$ with the stabilization of a certain module category, or alternatively, as a suitable category of enriched lifts. To this end, we will follow the outline described in the introduction. Recall that associated to a $\P$-algebra $A$ is its enveloping operad $\P^A$ (see~\S\ref{ss:prelim}), whose characteristic property is a natural equivalence of categories $\Alg_{\P^A} \simeq (\Alg_{\P})_{A/}$. Under this equivalence, the identity map $A \lrar A$ exhibits $A$ as the initial $\P^A$-algebra, so that $\Alg_{\P^A}^{\aug} \simeq (\Alg_{\P})_{A//A}$. We may hence write the tangent model category at $A$ as
$\T_A\Alg_{\P} \simeq \Sp(\Alg_{\P^A}^{\aug})$. Under the conditions of Theorem~\ref{t:comp} we now obtain a right Quillen equivalence $\T_A\Alg_{\P}\x{\sim}{\lrar} \Sp(\Alg_{\P^A_{\leq 1}}^{\aug})$. The category $\Alg_{\P^A_{\leq 1}}$ is just the category $(\Mod^{\P}_A)_{A/}$ of $A$-modules in $\M$ under $A$ (see Remark \ref{r:p01}). This leads to the following corollary:

\begin{cor}\label{c:sum-comparison}
Let $\M$ be a differentiable, left proper, combinatorial SM model category and let $\P$ be an admissible $\Sig$-cofibrant operad. Let $A$ be a $\P$-algebra and assume either that $A$ is fibrant or that $\M$ is right proper. In addition, assume that \textbf{at least one} of the following conditions holds:
\begin{enumerate}[(1)]
\item
$\Alg_\P$ is left proper and $\P^A$ is $\Sig$-cofibrant.
\item
$A$ is cofibrant and the stable model structure $\Sp((\Alg_\P)_{A//A})$ exists.
\end{enumerate}
Then restriction along $\vphi: \P^A_{\leq 1}\lrar \P^A$ yields a right Quillen equivalence
\begin{equation}\label{e:equiv-tangent}
\xymatrix@R=0em{
\vphi^*_{\Sp}: \T_A\Alg_\P \simeq \Sp(\Alg_{\P^A}^{\aug})
\ar^-{\sim}[r]  &  
\Sp(\Alg^{\aug}_{(\P^A)_{\leq 1}}) \simeq   \T_A\Mod^\P_A.
} 
\end{equation} 
\end{cor}
\begin{proof}
We apply Theorem~\ref{t:comp} to the operad $\P^A$. For this we need to check that in both cases (1) and (2) the operad $\P^A$ is stably admissible and $\Sig$-cofibrant. In case (1) the stable model structure $\Sp((\Alg_\P)_{A//A})$ exists because $\Alg_\P$ is left proper and in case (2) it is simply assumed. Similarly, in case (1) we assume that $\P^A$ is $\Sig$-cofibrant. It will hence suffice to check that $\P^A$ is also $\Sig$-cofibrant under the conditions of (2), which holds by~\cite[Proposition 2.3]{BM09}.
\end{proof}

We may remove the conditions that $\M$ is left proper and that the stable model structure $\Sp((\Alg_\P)_{A//A})$ exists and instead consider the relative category $\T_A'\Alg_{\P} = \Sp'((\Alg_{\P})_{A//A})$ from Definition \ref{d:sp'} (see~\cite[Remark 3.3.4]{part1}). Replacing Theorem~\ref{t:comp} by Corollary~\ref{c:compmodel}, we then obtain the following: 

\begin{cor}\label{c:sum-comparison-model}
Let $\M$ be a differentiable, combinatorial SM model category and let $\P$ be an admissible $\Sig$-cofibrant operad. Let $A$ be a $\P$-algebra and assume either that $A$ is fibrant or that $\M$ is right proper. In addition, assume that one of the following conditions holds:
\begin{enumerate}[(1)]
\item
$\P^A$ is $\Sig$-cofibrant.
\item
$A$ is cofibrant.
\end{enumerate}
Then restriction along $\vphi: \P^A_{\leq 1}\lrar \P^A$ yields an equivalence of relative categories
$$
\vcenter{\xymatrix{
\vphi^*_{\Sp}: \T_A'\Alg_{\P} = \Sp'(\Alg_{\P^A}^{\aug}) \ar^-{\sim}[r]  &  
 \Sp'(\Alg^{\aug}_{(\P^A)_{\leq 1}}) = \T_A'\Mod^\P_A.
}} 
$$
\end{cor}

\begin{rem}\label{r:fresse-2}
When every object in $\M$ is cofibrant and $\P$ is a cofibrant single colored operad (with respect to the transferred model structure on operads in $\M$), then $\Alg_\P$ is left proper and $\P^A$ is $\Sig$-cofibrant for every $\P$-algebra $A$ \cite{fresse} (see also Remark~\ref{r:fresse}). This is also true when $\P$ is a cofibrant colored operad and $\M$ is the model category of simplicial sets by \cite{rezk}.
\end{rem}

We may use Remark~\ref{r:global-tensored} to rewrite the right hand side of~\eqref{e:equiv-tangent} as the full subcategory $\Fun^{\M}_{/\M}(\P^A_1,\T\M) \subseteq \Fun^{\M}(\P^A_1,\T\M)$ consisting of those enriched functors $\F: \P^A_1 \lrar \T\M$ which lie above the functor $\P^A_0: \P^A_1 \lrar \M$.
\begin{cor}\label{c:sum-comparison-2}
Let $\M, \P$ and $A$ be as in Corollary~\ref{c:sum-comparison} and assume in addition that $\M$ is tractable. Then we have a natural right Quillen equivalence
$$ \T_A\Alg_\P \x{\sim}{\lrar} \Fun^{\M}_{/\M}(\P^A_1,\T\M)  .$$
\end{cor}

When $\M$ is stable and strictly pointed the situation simplifies. 
\begin{cor}\label{c:sum-comparison-stable}
Let $\M, \P$ and $A$ be as in Corollary~\ref{c:sum-comparison} and assume in addition that $\M$ is stable and strictly pointed. Let $\K: \Alg^\P_{A//A} \lrar (\Mod^\P_A)_{A//A} \x{\ker}{\lrar} \Mod^\P_A$ be the composition of the forgetful functor and the kernel functor appearing in Lemma~\ref{l:ker-stable}. Then the functors
$$ \xymatrix{
\T_A\Alg_\P \ar[r]^-{\K_{\Sp}}_-{\sim} & \Sp(\Mod^{\P}_A) \ar[r]^-{\Om^{\infty}}_-{\sim} & \Mod^\P_A \\
}$$
are right Quillen equivalences.
\end{cor}
\begin{proof}
The category $\Mod^{\P}_A(\M) \simeq \Fun^{\M}(\P^A_1,\M)$, endowed with the projective model structure, is stable, strictly pointed and left proper, because $\M$ has these properties and $A$ is cofibrant in $\M$ (by either Condition (1) or (2) of Corollary~\ref{c:sum-comparison}).
In particular, $\Om^{\infty}: \Sp(\Mod^\P_A) \lrar \Mod^\P_A$ is a right Quillen equivalence (see \cite[Corollary 3.3.3]{part1}). If $\M$ is right proper then $\Mod^\P_A(\M)$ is right proper and if $A$ is fibrant as an algebra then $A$ is fibrant as an $A$-module. Lemma~\ref{l:ker-stable} then implies that $\ker:\Mod^{\P}_A(\M)_{A//A} \x{\sim}{\lrar} \Mod^{\P}_A(\M)$ is a right Quillen equivalence. Combining this with Corollary~\ref{c:sum-comparison} we may now conclude that $\K_{\Sp}$ is a right Quillen equivalence. 
\end{proof}

\subsection{The $\infty$-categorical comparison}\label{s:infinity}
Our goal in this section is to formulate and prove an $\infty$-categorical counterpart of Corollary~\ref{c:sum-comparison}. For this it will be useful to consider another approach for the theory of modules, where one considers the collection of pairs $(A,M)$ of a $\P$-algebra $A$ and an $A$-module $M$ as algebras over another operad $\M\P$. We shall henceforth follow the approach of~\cite{Hin15}. Let $\Com$ be the commutative operad and let $\M\Com$ be the operad for commutative algebras and modules over them (see Example \ref{e:basic-1}(\ref*{e:basic-1mcom})). There are natural maps $\Com \lrar \M\Com \lrar \Com$ where the first one sends the only object of $\Com$ to $a$ and the second is the terminal map. Restriction along $\Com \lrar \M\Com$ induces the projection $(A,M) \mapsto A$ which forgets the module.

Given a simplicial operad $\P$ we will denote by 
$$ \M\P := \M\Com \times_{\Com} \P $$
the associated fiber product in the category of simplicial operads. If $\C$ is a simplicial model category, then an $\M\P$-algebra in $\C$ is the same as a pair $(A,M)$ where $A$ is a $\P$-algebra in $\C$ and $M$ is an $A$-module.

We will denote by $\M\Com^{\otimes} := \rN^{\otimes}(\M\Com)$ the operadic nerve of $\M\Com$. Given an $\infty$-operad $\O^{\otimes}$ we will denote by
$$ \M\O^{\otimes} := \M\Com^{\otimes} \times_{\Com^{\otimes}} \O^{\otimes} $$
the associated (homotopy) fiber product in the model category of pre-operads. Since the operadic nerve preserves fiber products we have that if $\P$ is a simplicial operad then $\rN^{\otimes}(\M\P) \cong \M\rN^{\otimes}(\P)$.
\begin{define}[{\cite[Def. 5.2.1]{Hin15}}]
Let $\O^{\otimes}$ be an $\infty$-operad and $\C^{\otimes}$ a symmetric monoidal $\infty$-category. Let $A \in \Alg_{\O}(\C)$ be an $\O$-algebra object in $\C$. The $\infty$-category $\Mod^{\O}_A(\C)$ is defined as the fiber product
$$ \Mod^{\O}_A(\C) := \Alg_{\M\O}(\C) \times_{\Alg_{\O}(\C)}\{A\}. $$
\end{define}

We will refer to $\Mod^{\O}_A(\C)$ as the $\infty$-category of $A$-modules in $\C$. When the $\infty$-operad $\O$ is unital and coherent, Proposition B.1.2 in \cite{Hin15} establishes a natural equivalence of $\infty$-categories from $\Mod^{\O}_A(\C)$ to the underlying $\infty$-category of the $\O$-monoidal $\infty$-category $\Mod^{\O}_A(\C)^\otimes$ of $A$-modules of \cite[\S 3.3.3]{Lur14}. Furthermore, the following variation on the arguments of \cite{Hin15} shows how such $A$-modules in the $\infty$-categorical sense can be strictified.

\begin{pro}\label{p:hin}
Let $\M$ be a combinatorial simplicial SM model category and let $\P$ be a $\Sig$-cofibrant admissible simplicial operad such that $\M\P$ is admissible as well. For any cofibrant $A$ in $\Alg_{\P}(\M)$, there is an equivalence of $\infty$-categories
$$\xymatrix{
\Mod_A^{\P}(\M)_\infty \ar[r]^-\sim & \Mod_A^{\rN(\P)}(\M_\infty).
}$$
\end{pro}
 \begin{proof}
If $\P$ is $\Sig$-cofibrant and admissible, then the associated simplicial operad $\M\P$ is $\Sig$-cofibrant and admissible as well. By~\cite[Theorem 7.11]{PS18}, the map of operads $\P\lrar \M\P$, obtained as the base change of the map $\Com\lrar \M\Com$, induces a commuting square of $\infty$-categories
\begin{equation}\label{d:mods}\vcenter{\xymatrix@R=1.3pc@C=1.3pc{
\Alg_{\M\P}(\M)_\infty\ar[r]^-\sim\ar[d]_{p} & \Alg_{\M\rN(\P)}(\M_\infty)\ar[d]^q\\
\Alg_{\P}(\M)_\infty\ar[r]_-\sim & \Alg_{\rN(\P)}(\M_\infty),
}}\end{equation}
in which the horizontal maps are equivalences of $\infty$-categories. Now observe that the left vertical map $p$ of $\infty$-categories is obtained by localization from the functor of relative categories
$$
\pi :\Alg_{\M\P}(\M)'\lrar  \Alg_{\P}(\M)^{\cof},
$$
whose domain $\Alg_{\M\P}(\M)'$ is the relative subcategory of $\Alg_{\M\P}(\M)$ on those pairs $(A, M)$ of algebras and modules whose algebra $A$ is cofibrant. To see that the $\infty$-category $\Alg_{\M\P}(\M)'_\infty$ is equivalent to $\Alg_{\M\P}(\M)_\infty$, note that $\Alg_{\M\P}(\M)^{\cof}$ is a relative subcategory of $\Alg_{\M\P}(\M)'$ 
and that the inclusion $\Alg_{\M\P}(\M)^{\cof}\lrar \Alg_{\M\P}(\M)'$ is part of a left homotopy deformation retract, with retraction given by a cofibrant replacement functor in $\Alg_{\M\P}(\M)$.

We may now identify $\Alg_{\M\P}(\M)'$ with the Grothendieck construction of the functor $\Mod^{\P} : \left(\Alg_{\P}(\M)^{\cof}\right)^{\op}\lrar \RelCat$ sending a cofibrant $\P$-algebra $A$ to the relative category $\Mod_A^{\P}(\M)$ of $A$-modules and a map $f : A\lrar B$ of cofibrant $\P$-algebras to the restriction functor $f^*: \Mod^{\P}_B(\M)\lrar \Mod_A^{\P}(\M)$ between module categories. We note that the functor $\Mod^{\P}$ sends weak equivalences of cofibrant algebras to equivalences of relative categories by \cite[Theorem 2.6]{BM09}. We may hence apply~\cite[Proposition 2.1.4]{Hin} to the map $\pi^{\op}$ and deduce that for every cofibrant $\P$-algebra $A$ we have a chain of equivalences of $\infty$-categories
$$
\Mod_A^{\P}(\M)_\infty \simeq \pi^{-1}(A)_\infty  \x{\simeq}{\lrar} p^{-1}(A)  \x{\simeq}{\lrar} q^{-1}(A) \simeq\Mod^{\rN(\P)}_A(\M_\infty),
$$
where the second map is the induced map on fibers arising from \eqref{d:mods}, and thus an equivalence.
\end{proof}

\begin{thm}\label{t:compoo}
Let $\C$ be a closed SM, differentiable presentable $\infty$-category and let $\O^{\otimes} := \rN^{\otimes}(\P)$ be the operadic nerve of a fibrant simplicial operad and let $A \in \Alg_{\O}(\C)$ be an $\O$-algebra. Then the forgetful functor induces an equivalence of $\infty$-categories
$ \T_A\Alg_{\O}(\C) \x{\simeq}{\lrar} \T_{A}\Mod^{\O}_A(\C) $.
\end{thm}

\begin{proof}
Since weakly equivalent fibrant simplicial operads have equivalent associated $\infty$-operads, we may assume that $\P$ is $\Sig$-cofibrant. By \cite[Theorem 1.1]{NS} there exists a left proper, combinatorial simplicial SM model category $\M$ together with a symmetric monoidal equivalence of $\infty$-categories $(\M^\otimes)_\infty \simeq \C^\otimes$. Furthermore, $\M$ has the property that any simplicial operad is admissible \cite[Theorem 2.5]{NS}. Since $\C$ is assumed to be differentiable, the model category $\M$ is differentiable as well. By the rectification result of~\cite[Theorem 7.11]{PS18} the model category $\Alg_\P(\M)$ presents the $\infty$-category $\Alg_{\O}(\C)$. Let $\ovl{A}$ be a fibrant-cofibrant $\P$-algebra in $\M$ representing $A$. Then the slice-coslice model structure $(\Alg_\P(\M))_{\ovl{A}//\ovl{A}}$ presents the $\infty$-category $(\Alg_{\O}(\C))_{A//A}$ and by Proposition~\ref{p:hin} the transferred model structure on $\Mod_{\ovl{A}}^{\P}(\M)$ presents the $\infty$-category $\Mod_A^{\O}(\C)$. We note that since $\P$ is $\Sig$-cofibrant and $\ovl{A}$ is cofibrant, $\P^{\ovl{A}}$ and $\P^{\ovl{A}}_{1}$ are $\Sigma$-cofibrant by \cite[Proposition 2.3]{BM09}. This means, in particular, that the forgetful functor $\Mod^\P_{\ovl{A}} \lrar \M$ is a left Quillen functor and so $\Mod^\P_{\ovl{A}}$ inherits from $\M$ the property of being left proper (see Remark \ref{r:admissible}). In particular, the slice-coslice model structure $(\Mod^\P_{\ovl{A}}\M)_{\ovl{A}//\ovl{A}}$ presents the $\infty$-category $(\Mod^{\O}_{A}(\C))_{A//A}$.
Consider the commutative diagram of $\infty$-categories
$$\xymatrix@R=1.3pc@C=1.3pc{
\Sp'\left(\Alg_\P(\M)_{\ovl{A}//\ovl{A}}\right)_\infty \ar^{\simeq}[r] \ar[d] &\Sp\left((\Alg_\P(\M)_\infty)_{A//A}\right)\ar^{\simeq}[r]\ar[d]  &  \Sp\left(\Alg_{\O}(\C)_{A//A}\right)\ar[d]\\
\Sp'\left(\Mod_{\ovl{A}}^{\P}(\M)_{\ovl{A}//\ovl{A}}\right)_\infty \ar^{\simeq}[r] & \Sp\left((\Mod_{\ovl{A}}^{\P}(\M)_\infty)_{A//A}\right)\ar^{\simeq}[r] & \Sp\left(\Mod_A^{\O}(\C)_{A//A}\right),
}$$
where for a model category $\N$ we denote by $\Sp'(\N) \subseteq \N^{\NN \times \NN}$ the full relative subcategory spanned by the $\Om$-spectra (Definition \ref{d:sp'}). Now the horizontal maps in the right square are equivalences of $\infty$-categories by construction and 
the horizontal maps in the left square are equivalences by~\cite[Remark 3.3.4]{part1}. Finally, the left vertical map is an equivalence by Corollary~\ref{c:sum-comparison-model}. It then follows that the right vertical map is an equivalence, as desired.
\end{proof}
When $\C$ is stable, the $\infty$-category $\Mod_A^{\O}(\C)$ is stable and the kernel functor (cf. Lemma \ref{l:ker-stable}) yields an equivalence of $\infty$-categories $\T_A\Mod_A^{\O}(\C) \simeq \Mod_A^{\O}(\C)$ for every $A$. In this case the conclusion of Theorem~\ref{t:compoo} reduces to the following generalization of~\cite[Theorem 7.3.4.13]{Lur14} to the case of $\infty$-operads which are not necessarily unital or coherent (but which do arise as nerves of simplicial operads):
\begin{cor}
Let $\C$ be a closed SM, \textbf{stable} presentable $\infty$-category and let $\O^{\otimes} := \rN^{\otimes}(\P)$ be the operadic nerve of a fibrant simplicial operad. Then the functor $\ker: \Alg_{\O}(\C)_{A//A} \lrar  \Mod^{\O}_A(\C)$ induces an equivalence of $\infty$-categories
$$ \T_A\Alg_{\O}(\C) \x{\simeq}{\lrar} \Mod^{\O}_A(\C) .$$
\end{cor}
Note that a stable presentable $\infty$-category $\C$ is always differentiable since $\colim: \C^{\NN}\lrar \C$ is an exact functor \cite[Example 6.1.1.7]{Lur14}.
\begin{rem}\label{r:rectalg}
It seems likely that Theorem \ref{t:compoo} admits a generalization to the case of enriched $\infty$-operads, whose theory has been developed in \cite{CH17}. By \cite[Corollary 5.1.9]{CH17}, the $\infty$-category of $W$-colored $\infty$-operads enriched in a SM $\infty$-category $\C$ is equivalent to the $\infty$-category of algebras in $\C$ over a certain operad $\Op_W$ (in sets!). In particular, if $\M$ is a SM model category, then every $W$-colored $\M$-enriched operad determines a $W$-colored $\infty$-operad enriched in $\M_\infty$. Conversely, every $\M_\infty$-enriched $\infty$-operad can be rectified to an $\M$-enriched operad when $\M$ is sufficiently nice \cite[Corollary 5.2.7]{CH17}.

To generalize Theorem \ref{t:compoo} to this enriched setting, one needs to generalize these results to algebras over operads as well. This can likely be done using an argument analogous to that of Proposition~\ref{p:hin}: there is an operad $\Q$ (in sets) for $W$-colored operads together with an algebra over them. One can then use the rectification machinery of \cite{PS18} to compare the fibers of the forgetful functors $\Alg_\Q(\M)\lrar \Alg_{\Op_W}(\M)$ and $\Alg_{\Q}(\M_\infty)\lrar \Alg_{\Op_W}(\M_\infty)$.
\end{rem}

\appendix

\section{Sifted homotopy colimits of algebras}
The purpose of this appendix is to prove the following result:
\begin{pro}\label{p:sifted}
Let $\M$ be a combinatorial SM model category and let $\P$ be a $\Sig$-cofibrant admissible $W$-colored operad in $\M$. Then the forgetful functor $U: \Alg_{\P}(\M)\lrar \M^{W}$ preserves sifted homotopy colimits.
\end{pro}

A close variant of Proposition~\ref{p:sifted} appears in~\cite[Proposition 7.9]{PS18}. 
We emphasize that by sifted homotopy colimits we mean homotopy colimits of diagrams indexed by a category whose nerve is sifted as an $\infty$-category (see~\cite[\S 5.5.8]{Lur09}). We shall refer to such categories as \textbf{homotopy sifted}. This condition can also be phrased as saying that the diagonal map $\I \lrar \I \times \I$ is cofinal (in the $\infty$-categorical sense).
A typical example of a homotopy sifted category is a category which admits finite coproducts. In fact, to prove Proposition~\ref{p:sifted} it will be convenient to first reduce to this special case.
 
For a small category $\I$, let us denote by $\I^{\coprod}$ the category obtained from $\I$ by freely adding finite coproducts. Explicitly, we may identify $\I^{\coprod}$ with the full subcategory of $\Fun(\I^{\op},\Set)$ spanned by those presheaves which are finite coproducts of representables. There is a canonical fully-faithful inclusion $\I \hrar \I^{\coprod}$ given by the Yoneda embedding. We note that $\I^{\coprod}$ is always sifted since it admits finite coproducts. In addition, we have the following observation:
\begin{lem}\label{l:cofinal}
If $\I$ is homotopy sifted then $\I \hrar \I^{\coprod}$ is cofinal (in the $\infty$-categorical sense).
\end{lem}
\begin{proof}
We need to show that if $\F = R_{i_1} \coprod ... \coprod R_{i_n}: \I^{\op} \lrar \Set$ is a coproduct of representables then the comma category $\I_{\F/}$ is weakly contractible. We note that the projection $\I_{\F/} \lrar \I$ is a left fibration classified by the Cartesian product of functors $h_{i_1} \times ... \times h_{i_n}: \I \lrar \Set$, where $h_{i_k}$ denotes the functor \textbf{copresented} by $i_k$. 
Recall that the classifying space of the domain of a left fibration 
is a model for the homotopy colimit of the associated functor; this follows immediately from the Bousfield-Kan formula for homotopy colimits. Since $\I$ is homotopy sifted, this homotopy colimit commutes with finite products (\cite[Lemma 5.5.8.11]{Lur09}). We may hence conclude that
$$ 
|\I_{\F/}| \simeq \hocolim_\I\Big(\prod_k h_{i_k}\Big) \simeq \prod_k \hocolim_\I h_{i_k} \simeq \ast 
$$
as desired. 
\end{proof}

\begin{proof}[Proof of Proposition~\ref{p:sifted}]
Let $\I$ be a homotopy sifted category. We wish to show that the forgetful functor $U:\Alg_\P \lrar \M^W$ preserves $\I$-indexed homotopy colimits. Since $\I \hrar \I^{\coprod}$ is fully-faithful, every $\I$-diagram is isomorphic to a diagram restricted from $\I^{\coprod}$. By Lemma~\ref{l:cofinal} the map $\I \hrar \I^{\coprod}$ is cofinal, and hence restriction along $\I^{\triangleright} \hrar (\I^{\coprod})^{\triangleright}$ preserves and detects homotopy colimit diagrams. To show that $U$ preserves $\I$-indexed homotopy colimits it will hence suffice to show that $U$ preserves $\I^{\coprod}$-indexed homotopy colimits. In other words, we may assume without loss of generality that $\I$ has finite coproducts. 

When $\I$ has finite coproducts the projective model structure on $\M^{\I}$ is \textbf{monoidal} with respect to levelwise tensor product. Indeed, since weak equivalences are levelwise and $\M$ is a monoidal model category it suffices to check that projective cofibrations satisfy the pushout-product axiom. This can be checked on generating projective cofibrations, which are of the form $h_i \otimes f$ for some generating cofibration $f$ in $\M$. One can then check that $(h_i \otimes f) \Box (h_j \otimes g) = h_{i\coprod j} \otimes (f \Box g)$ is indeed a projective cofibration.

We now proceed as in the proof of~\cite[Proposition 7.9]{PS18}. Identifying the underlying categories $(\Alg_\P)^{\I} \simeq \Alg_\P(\M^\I)$ we see that the projective model structure on the former coincides with the one transferred from $\M^\I_{\proj}$ on the latter. Since $\P$ is $\Sig$-cofibrant, \cite[Proposition 6.1.5]{WY18} implies that $U$ preserves projectively cofibrant $\I$-diagrams. Since $U$ preserves strict $\I$-colimits (being the forgetful functor from a category of algebras) we may now conclude that $U$ preserves $\I$-indexed homotopy colimits, as desired.
\end{proof}

\end{document}